%% file: main.tex
\documentclass[10pt]{tsoart}

\begin{document}

\title[Higher smooth structure sets of $\CC P^n$]{Higher smooth surgery structure sets \\ of complex projective spaces, part I}

\author{Samuel Kalu\v zn\'y}
\author{Tibor Macko}



\address{Faculty of Mathematics, Physics, and Informatics, Comenius University, Mlynsk\'a dolina,
SK-824 48, Bratislava, Slovakia} \email{tibor.macko@fmph.uniba.sk} \email{samuel.kaluzny@fmph.uniba.sk} 

\address{Institute of Mathematics, Slovak Academy of Sciences, \v Stef\'anikova 49, Bratislava, SK-81473, Slovakia} \email{macko@mat.savba.sk}

\thanks{This work was supported by grants VEGA 1/0425/25, UK/1192/2025, UK/3116/2024 and UK/197/2023. Parts of this work will be included in the PhD thesis of S.K} 

\date{}

\begin{abstract}
      This is the first of the two articles where we determine the higher smooth surgery structure sets of complex projective spaces (up to some extension problems) and the forgetful map to their topological versions in low dimensions. In this part, we concentrate on the free subgroup, where we obtain information in all dimensions. In the second part, we study the torsion.
\end{abstract}

\maketitle

\section{Introduction}
Let  $\CAT = \DIFF, \PL$ or $\TOP$. Given a closed $\CAT$-manifold $X$, the $\CAT$-surgery structure set $\sS^{\CAT} (X)$ is a certain set, recalled in detail in Definition~\ref{def:cat-structure-set}, where closed $\CAT$-manifolds homotopy equivalent to $X$ are organized. Varying $\CAT$, one obtains various forgetful maps between these structure sets. The principal aim of surgery theory is to determine these structure sets as well as the corresponding forgetful maps for a given $X$; thus obtaining important information about the $\CAT$-classification of manifolds in the homotopy type of $X$. We concentrate on the forgetful map $F_X \co \sS^{\DIFF} (X) \ra \sS^{\TOP} (X)$ as the difference between $\TOP$ and $\PL$ is small. A lot of calculations for specific $X$ have been done, but they are mostly confined to the case $\CAT = \TOP$, with only a handful of them when $\CAT= \DIFF$, see e.g. \cite[Ch.~18]{Lueck-Macko(2024)}, \cite{Chang-Weinberger(2021)}, \cite{Wall(1999)}. The situation is further complicated by the fact that while $\sS^{\TOP} (X)$ has a natural abelian group structure, there is no such natural structure on $\sS^{\DIFF} (X)$, see \cite[Sec. 11.8]{Lueck-Macko(2024)}, and hence $F_X$ is not a homomorphism in general. 

If $X$ is a compact $\CAT$-manifold with boundary, there are analogous structure sets $\sS^{\CAT}_{\partial} (X)$. If $X = Y \times \text{D}^k$ with $k \geq 1$, the structure sets $\sS^{\CAT}_{\partial} (Y \times \text{D}^k)$  become groups (abelian if $k \geq 2$) for any $\CAT$, and $F_X$ becomes a homomorphism. Hence, there is some hope that calculations might become more tractable in this case. Also, these groups provide information about (block) automorphism spaces of $Y$; we have $\pi_k (\G(Y)/\widetilde{\CAT} (Y)) \cong \sS_{\partial}^{\CAT} (Y \times \text{D}^k)$ for $k > 0$, see \cite[1.1.4]{Weiss-Williams(2001)}. In Remark~\ref{rem:surgery-motivation}, we elaborate on further motivation. 

The aim of this article (and its follow-up \cite{Kaluzny-Macko-2026-part-II}) is to present some calculations of $\sS^{\DIFF}_{\partial} (Y \times \text{D}^k)$ and $F_{Y \times \text{D}^k}$ when $Y = \CC P^n$. The results are complete (up to extension problems) for low values of $n$ and $k$. Moreover, we provide formulas and algorithms that, in principle, yield complete descriptions of the free part of the structure sets for all values of $n$ and $k$. Note that the calculation of $\sS^{TOP}_{\partial} (\cpd)$ is well-known through standard surgery methods, \cite[14C]{Wall(1999)}. Other previous work on this topic was done when $k=0$ by Brumfiel (see~\cite{Brumfiel1971}, and the unpublished~\cite{Brumfiel-(1970))}), who calculated (up to extension problems) the structure sets $\sS^{\DIFF} (\CC P^n)$ for $n \leq 6$, and by Little, see \cite{Little1989}, who, using Brumfiel's results, described the restriction of the forgetful map to the free part $F_{\CC P^n} \colon \sS^{\DIFF} (\CC P^n)_{\text{free}} \rightarrow \sS^{\TOP} (\CC P^n)_{\text{free}}$ when $n \leq 6$. Further related work has recently been done in~\cite{kasilingam2026diffeomorphismclassificationsmoothstructures} and~\cite{Basu-Kasilingam-Sarkar(2025)}, where smooth structures on $\cp^n$ and $\CC P^n \times S^k$ for low values of $n,k$ were studied. We comment on this relation in more detail in Remark~\ref{rem:surgery-motivation}.

We split our work into two parts, which correspond to different methods that are used. Roughly speaking, in this first part, we concentrate on the free part of $\textstyle\sS_{\partial}^{DIFF}(\cpd)$, which is mostly confined to studying what happens in dimensions divisible by $4$. The second part is devoted mainly to the torsion part of $\textstyle\sS_{\partial}^{DIFF}(\cpd)$, where the behavior in dimension $2$ modulo $4$ is prominent. More details on this separation are provided in Remark~\ref{rem:methods-in-two-parts}.


\section{The Main Theorems} \label{sec:main-theorems}


There are several versions of the surgery $\CAT$-structure sets; see~\cite[Ch. 11]{Lueck-Macko(2024)}. Since we will work only with simply-connected compact manifolds with boundary, we use the following one:

\begin{defn} \label{def:cat-structure-set}
    The \emph{relative $\CAT$-structure set} $\sS^{\CAT}_{\partial} (X)$ of a compact $\CAT$-manifold $X$ with boundary $\partial X$ is defined as the pointed set of equivalence classes of homotopy equivalences $(f,\partial f) \colon (M,\partial M)\rightarrow(X,\partial X)$, with $M$ being another compact $\CAT$-manifold with boundary, such that $\partial f$ is a $\CAT$-isomorphism. The equivalence relation is given by a diagram
    \begin{equation*}
        \begin{tikzcd}
            (M_1,\partial M_1) \arrow{d}[']{(g,\partial g)}{\cong} \arrow{r}{(f_1,\partial f_1)} & (X, \partial X) \\
            (M_2,\partial M_2) \arrow{ru}[']{(f_2,\partial f_2)} & 
        \end{tikzcd}
    \end{equation*}
    where $g \colon M_1 \rightarrow M_2$ is a $\CAT$-isomorphism, and which commutes up to homotopy rel $\partial M_1$. The basepoint is represented by $\id \colon X \rightarrow X$. 
    
    If $\partial X = \emptyset$, we write $\sS^{\CAT} (X)$. If $\partial X \neq \emptyset$, the \emph{absolute $\CAT$-structure set} $\sS^{\CAT} (X)$ is defined by dropping the condition that $\partial f$ is a $\CAT$-isomorphism and only requiring that $(f, \partial f)$ is a homotopy equivalence of pairs.
\end{defn}
(Note that, in general, the equivalence relation is given by $h$-cobordism. However, in the simply-connected situation all homotopy equivalences are simple, therefore all $h$-cobordisms are $s$-cobordisms, and in dimensions $\geq5$ we have that $s$-cobordism implies $\CAT$-isomorphism.)

When $X = Y \times \text{D}^k$, the pointed sets $\sS^{\CAT}_{\partial} (Y \times \text{D}^k)$ become groups (abelian if $k \geq 2$), where the group operation is given by ``stacking''; see~\cite[Sec. 11.8]{Lueck-Macko(2024)}. 
Our main results in this paper are the following three theorems and their corollary.

\begin{thm}
\label{mainthm1}
    When $2n+k\geq5$ there are isomorphisms 
    \begin{align*}
        \textstyle\sS_{\partial}^{DIFF}(\cpd)\cong\Z^{t_{n,k}}\oplus T_{n,k} 
    \end{align*}
    where
    \begin{align*}
        t_{n,k}=
        \begin{cases}
            0 & \text{if $k$ is odd}\\
            \lfloor\frac{n}{2}\rfloor+1 & \text{if $k\equiv0\mod4$ and $n$ is odd} \\
            \lfloor\frac{n}{2}\rfloor & \text{otherwise}
        \end{cases}
    \end{align*}
    and $T_{n,k}$ is a finite abelian group (the torsion subgroup). 
\end{thm}
It is a standard result (see \cite[14C]{Wall(1999)} for the case $l=0$, the general case follows analogously) that there are isomorphisms 
\begin{align} \label{eqn:top-str-set-of-cp-n-times-d-k}
    \textstyle\sS_{\partial}^{TOP}(\cpdt)\xrightarrow{\cong}
        \Z^{t_{n,2l}}\oplus\Z_2^{n-t_{n,2l}} \quad \textup{and} \quad \sS_{\partial}^{TOP} (\cp^n\times \text{D}^{2l+1})\xrightarrow{\cong} 0 
\end{align}
given by the so-called splitting invariants $\overline{\sigma}_{m,2l}$, which are surgery obstructions along the standard embedding of $\cp^m\times \text{D}^{2l}$ for $0\leq m< n$. 

The isomorphisms in Theorem~\ref{mainthm1} involve some choices: the choice of a specific splitting in Lemma~\ref{suspensg/o}, of generators of the solutions of systems \eqref{eqn:kerJ-matrix1}, \eqref{eqn:kerJ-matrix2} and \eqref{eqn:kerJ-matrix3}, and of some generators of the subgroup of roots of \eqref{eqn:obstr-final-formula-even-l}. The following two theorems depend on the latter two choices; the specific splitting is irrelevant. Suppose that we have made the required choices. We obtain a diagram
\begin{equation}
\label{eqn:diagram-forgetful-map}
    \begin{tikzcd}[column sep=2cm]
        \sS_{\partial}^{DIFF}(\cpdt) \arrow{d}{F_{\cpdt}} & \Z^{t_{n,2l}}\oplus T_{n,2l} \arrow{l}{\cong} \arrow[dashed]{d}{F_{n,l}} \\
        \sS_{\partial}^{TOP}(\cpdt) \arrow{r}{\cong} & \Z^{t_{n,2l}}\oplus\Z_2^{n-t_{n,2l}} 
    \end{tikzcd}
\end{equation}
where $F_{\cpdt}$ is the forgetful map and $F_{n,l}$ is given by a $(2\times2)$-block matrix
\begin{equation}
\label{eqn:matrix-forgetful-map}
    \begin{pmatrix}
        A_{n,l} & 0 \\
        B_{n,l} & C_{n,l}
    \end{pmatrix}.
\end{equation}
Here $A_{n,l}:\Z^{t_{n,2l}}\to\Z^{t_{n,2l}}$, $B_{n,l}:\Z^{t_{n,2l}}\to\Z_2^{n-t_{n,2l}}$ and $C_{n,l}:T_{n,2l}\to\Z_2^{n-t_{n,2l}}$. In the following two theorems, we determine the matrix $A_{n,l}$ (for odd $n+l$ up to an indeterminacy of index $2$). In general, it is given by a product $A_{n,l}^{'}\cdot P_{n,l}$. Suppose that the generators of solutions of \eqref{eqn:kerJ-matrix1}, \eqref{eqn:kerJ-matrix2} and \eqref{eqn:kerJ-matrix3} were chosen as in Remark~\ref{kerJgenerstandembed}. Then:
\begin{thm}
\label{mainthm2}
    The matrix $A_{n,l}^{'}$ has the following entries:
    \begin{align*}
        A_{n,l}^{'}=
        \begin{cases}
            (a_{2m,l}^i)_{0\leq m,i\leq t_{n,l}-1} & \text{if $l$ is even} \\
            (b_{2m+1,l}^{i+1})_{0\leq m,i\leq t_{n,l}-1} & \text{if $l$ is odd}
        \end{cases}
    \end{align*}
    with an additional column of zeros when $n+l$ is even, where the coefficients $a_{m,l}^i,b_{m,l}^i$ are zero for $m<i$ (the matrix is lower triangular) and are otherwise given by \eqref{eqn:coeff-a_n,l}, \eqref{eqn:coeff-a_n,l-sphere} and \eqref{eqn:coeff-b_n,l}.
\end{thm}
\begin{thm}
\label{mainthm3}
    If $n+l$ is even, the columns of $P_{n,l}$ are the chosen generators of the subgroup of roots of the linear polynomial \eqref{eqn:obstr-final-formula-even-l}. Otherwise, $P_{n,l}$ is either the identity matrix or $\coker P_{n,l}\cong\Z_2$.
\end{thm}
\begin{com}
    In \cite{Kaluzny-Macko-2026-part-II}, the follow-up to this article, we show that for some specific choice of splitting in Lemma~\ref{suspensg/o} $P_{n,l}$ is the identity matrix for all odd $n+l$. The indeterminacy in Theorem~\ref{mainthm3} is thus resolved. 
\end{com}
\begin{remark}
    We calculated the numerical values of entries in $A_{n,l}^{'}$ for $1\leq n,2l\leq6$ and $2(n+l)\geq6$ by applying a series of algorithms (coded in Python) to find the solutions of systems \eqref{eqn:kerJ-matrix1}, \eqref{eqn:kerJ-matrix2} and \eqref{eqn:kerJ-matrix3}, and the values of coefficients $a_{m,l}^i,b_{m,l}^i$ from \eqref{eqn:coeff-a_n,l}, \eqref{eqn:coeff-b_n,l} and \eqref{eqn:coeff-a_n,l-sphere}. These results are summarized in Table~\ref{tabobstrtors}.
\end{remark}

\begin{cor}
\label{maincoroll}
Let $(f,\partial f):(X,\partial X)\to(\cp^6\times D^{2l},\cp^6\times S^{2l-1})$ represent an element of $\sS^{TOP}(\cp^6\times D^{2l})$, $1\leq l\leq3$.
    \begin{enumerate}
        \item The map $(f,\partial f)$ admits a smooth structure on $X$ only if its integral splitting invariants $(\overline{\sigma}_{\epsilon,2l},\overline{\sigma}_{2+\epsilon,2l}\dots,\overline{\sigma}_{4+\epsilon,2l})$ (where $\epsilon=0$ if $l$ is even and $1$ if $l$ is odd) satisfy the congruences given in Table \ref{tabsplitinv}.
        \item If $[(f,\partial f)]\in\sS_{\partial}^{TOP}(\cpdt)$ is divisible by $2$ and its splitting invariants satisfy the congruences in Table~\ref{tabsplitinv}, then $(f,\partial f)$ admits a smooth structure on $X$.  
    \end{enumerate}
    Similar conditions can be obtained in dimensions $1\leq l\leq3$ and $0\leq n\leq5$ from the matrices in Table~\ref{tabobstrtors}, by the process described in the proof of this corollary at the end of Section~\ref{nonsmooth}.  
\end{cor}
\begin{com}
    The torsion subgroups $T_{n,k}$ for $1\leq n,k\leq6$, as well as a full description of $\im(F_{\cpdt})$ in corresponding dimensions will be given in~\cite{Kaluzny-Macko-2026-part-II}. This will lead to an improved version of the above corollary, where we obtain a full characterization of the image of $F_{\cpdt}$ in terms of congruences.
\end{com}

\begin{remark}[Methods and Outline] \label{rem:methods-in-two-parts}
In our work, we follow the standard surgery method for determining the structure sets via the surgery exact sequence; see \cite[Ch.~11]{Lueck-Macko(2024)}. Section~\ref{sec:surgery} contains a description of this method applied to our case. Since we are in the simply-connected situation, the $L$-groups are known, and it remains to determine the normal invariants and the surgery obstructions. To determine the smooth normal invariants, which we do in Section~\ref{norminvsect}, it is convenient to use their description as the group of homotopy classes of maps to the classifying space $G/O$, the homotopy fiber of the map $J \colon BO \rightarrow BG$, whose induced map on the homotopy groups is known as the classical $J$-homomorphism. We proceed by looking at the long exact sequence induced by the fibration obtained from $J$ which gives a splitting of the normal invariants of $X=\CC P^n \times \text{D}^k$ into the free and torsion parts; see \eqref{norminvsplitting} and Lemma~\ref{suspensg/o}. The free part is given by the kernel of the map induced by $J$, the torsion part by the cokernel of the map induced by $\Omega J$. The free part is directly related to vector bundle theory, whereas the torsion part is directly related to stable homotopy theory. These two parts are studied by different sets of tools. This is the reason for splitting our work into two papers. 

Much is known about the homotopy type of the space $G/O$; see \cite[Ch.~5]{Madsen}. However, to study the free part, we use instead older papers by Adams and Adams-Walker; \cite{AdamsI,Adams,AdamsIII,adamsIV,Adams_Walker_1965}. Using their work, it is possible to obtain not only a description of the free part of the normal invariants as an abelian group, but also a good description of its generators in terms of specific bundles over $\Sigma^k\cp^n\vee S^k\simeq\CC P^n \times \text{D}^{k}/\cp^n\times S^{k-1}$. By Lemma~\ref{suspensg/o} one needs to find criteria for vector bundles over $\Sigma^k\cp^n$ to be trivial as spherical fibrations. The above works, together with an observation from a newer paper~\cite{federgitler}, give such criteria formulated in terms of genera corresponding to various multiplicative sequences of characteristic classes, in the sense of~\cite{milnor1974characteristic}, see~\eqref{equationkerJ}. Furthermore, the fact that the space $\Sigma^k\cp^n$ is a suspension means that all cup products vanish, which simplifies the criterion from~\eqref{equationkerJ} to~\eqref{eqn:simplified-equationkerJ}. The main task then boils down to determining integer solutions to a certain specific system of linear equations, whose coefficients are rational numbers coming from the multiplicative sequences mentioned above; see Lemma~\ref{generkerJlema}. This is then implemented in Python to obtain numerical results for specific values of $l$ and $n$; some cases are listed in Tables \ref{kerJgenertable}, \ref{kerJgenertable-l-big-n-4} that appear in Appendix. 

The next step is to determine surgery obstructions, which we do in Section~\ref{obstrsect}, thus proving Theorem~\ref{mainthm1}. Here, it is convenient to view the normal invariants in terms of degree one normal maps. The standard correspondence between the two descriptions comes from a version of the Pontryagin-Thom construction. The surgery obstruction in our setting turns out to be given by the signature. Once we have this, we can employ the Hirzebruch signature formula which says that the signature equals the $\sL$-genus. (Its use has to be slightly adapted since we have manifolds with boundary; see Lemma~\ref{gluinglema}.) This leads to the formula~\eqref{eqn:obstr-formula-Hirz-simplified}. On the other hand, we have a description of the generators of the free part of the normal invariants from the previous section via specific bundles whose characteristic classes are the input for the formula. Here, the situation is also simplified by vanishing cup products. This allows us to calculate the necessary characteristic classes as well as the $\sL$-genus and obtain formulas in terms of linear polynomials~\eqref{eqn:obstr-final-formula-even-l}. Again, some numerical values are computed algorithmically and listed in Tables \ref{signobstrtable1}, \ref{signobstrtable2}, \ref{signobstrtable3}.   

It is interesting to compare this with Brumfiel's work \cite{Brumfiel-(1970))}, who studied the case where $k=0$ and $n \leq 6$. We do not find his approach generalizable to higher $k$ and $n$ in an obvious way. On the other hand, one could, in principle, proceed by our methods also for $k=0$ and higher $n$, but due to non-vanishing cup products, it would be computationally more complex. So, indeed, the indication that calculating higher structure sets $\textstyle\sS_{\partial}^{DIFF}(\cpd)$ rather than $\textstyle\sS^{DIFF}(\cp^n)$ might be easier seems to be confirmed.

In Section \ref{nonsmooth} we study splitting invariants which are obtained as surgery obstructions on submanifolds $\cp^m\times\text{D}^{2l}$. Here the proofs of Theorems~\ref{mainthm2}, \ref{mainthm3} and Corollary~\ref{maincoroll} are completed.

In the second part \cite{Kaluzny-Macko-2026-part-II} we study the torsion subgroup. The methods will involve an adaptation to $k > 0$ of the spectral sequence used by Brumfiel to determine the torsion subgroup of the normal invariants when $k=0$ (which itself contains many more homotopy theoretic calculations), and calculations of the surgery obstruction which in that case is given by the Arf invariant instead of the signature. 
\end{remark}

\begin{remark}[Motivation] \label{rem:surgery-motivation}
Besides the determination of the structure sets and the forgetful map, our motivation also comes from the following two sources. 

Firstly, recall that given $N \geq 2$, we have an $S^1$-bundle $p_N \co L^{2n+1}_N \ra \CC P^n$ where $L^{2n+1}_N$ is the standard lens space with $\pi_1 (L^{2n+1}_N) \cong \ZZ/N$. There are induced transfer maps $(p_N)^{!} \co \sS^{\CAT}_{\partial} (\CC P^n \times \text{D}^k) \ra \sS^{\CAT}_{\partial} (L^{2n+1}_N \times \text{D}^k)$ for all $k \geq 0$. Using our results, we hope to obtain information about the forgetful map for the lens spaces as well.

Secondly, there are certain results on topological surgery that are obtained using the abelian group structure on $\sS^{\TOP} (X)$. An example is the long exact sequence of abelian groups 
\[
	\cdots \rightarrow \sS^{\TOP}_{\partial} (X \times \text{D}^k) \rightarrow \sS^{\TOP} (X \times S^k)  \rightarrow \sS^{\TOP} (X \times \text{D}^k) \rightarrow \cdots
\]
for $k \geq 3$, where the first map is given by extending by homeomorphism and the second map is given by Browder's splitting theorem~\cite[Ch. 11]{Wall(1999)}. This sequence appears in~\cite[Proposition 7.2.6 (ii)]{Ranicki(1981)} as the sequence associated to the pair $(X,Z) = (X \times S^k,X \times \text{D}^k)$. As we recalled above, it is known that on $\sS^{\DIFF} (X)$ there is no analogous abelian group structure. On the other hand, it is still possible that some results, such as an analog of the above sequence, might hold even in the smooth case; or, if they do not, it would be interesting to know how much they fail. In the future, we would like to use our calculations, together with those from \cite{Basu-Kasilingam-Sarkar(2025)}, to investigate this question for $X = \CC P^n$.
\end{remark}

\input{surgery}
\input{norminv}
\input{obstruction}
\input{splittinginv}
\newpage
\bibliography{bibliography}  
\bibliographystyle{alpha}
\newpage
\appendix
\input{appendix}

\end{document}

%% file: surgery.tex
\section{Surgery-Theoretical Background}
\label{sec:surgery}

The calculation of the structure sets $\sS_{\partial}^{DIFF}(\cpd)$ begins with the surgery exact sequence \cite[Section 11.5]{Lueck-Macko(2024)}, which is a sequence of pointed sets and is defined for $2n+k\geq5$. In our case, it divides into 
\begin{align}
\label{SESeven}
    \textstyle0\to\sS_{\partial}^{DIFF}(\cpd)\to\sN_{\partial}^{DIFF}(\cpd)\xrightarrow{\sigma_{n,k}^{DIFF}}L_{2n+k}(\Z)    
\end{align}
if $2n+k$ is even and 
\begin{align}
\label{SESodd}
    \textstyle\sN_{\partial}^{DIFF}(\cp^n\times \text{D}^{k+1})\xrightarrow{\sigma_{n,k+1}^{DIFF}}L_{2n+k+1}(\Z)\to\sS_{\partial}^{DIFF}(\cpd) \\ \nonumber
    \to\sN_{\partial}^{DIFF}(\cpd)\to0    
\end{align}
if $2n+k$ is odd. Here, 
\[  
    L_i(\Z)\cong
    \begin{cases}
        \Z & \text{if }i\equiv0\mod4 \\
        \Z_2 & \text{if }i\equiv2\mod4 \\
        0 & \text{otherwise}
    \end{cases}
\]
The set of normal invariants $\sN_{\partial}^{DIFF}(\cpd)$ consists of equivalence classes of normal maps of degree one from smooth manifolds to $\cpd$, which are diffeomorphisms on the boundary. The equivalence relation is given by normal bordism. More details can be found at the beginning of Section~\ref{obstrsect}, as well as in \cite[Chapter 7]{Lueck-Macko(2024)}. The surgery obstruction map $\sigma_i^{DIFF}:\sN_{\partial}^{DIFF}(X^i)\to L_i(\Z)$ is given by the difference of signatures when $i\equiv0\mod 4$, and by the Arf invariant of an associated quadratic form when $i\equiv2\mod4$. Moreover, if $k>0$ (which we assume in this paper), the terms in the sequences (\ref{SESeven}) and (\ref{SESodd}) are groups (abelian if $k > 1$), and $\sigma_i^{DIFF}$ is a homomorphism. For more details, see \cite[Chapters~8,~9]{Lueck-Macko(2024)}, especially Sections 8.7.6 and 9.4.4. 

From (\ref{SESeven}) and (\ref{SESodd}), we see that to compute $\sS_{\partial}^{DIFF}(\cpd)$ it is sufficient to determine the normal invariants $\sN_{\partial}^{DIFF}(\cpd)$ for all $n,k$ and the map $\sigma_{n,k}^{DIFF}$ for even $2n+k$. In this paper, we concentrate on the case where $2n+k$ is divisible by $4$; in part II, we deal with the case where $2n+k$ is $2\mod4$.

%% file: norminv.tex
\section{The Normal Invariants}
\label{norminvsect}
In this section, we study the group $\sN_{\partial}^{\DIFF} (\cpd)$ when $k>0$. As discussed in Remark~\ref{rem:complexity-k-is-0-gener-ker-J}, in more detail, the computation is somewhat simplified compared to the case $k=0$ studied by Brumfiel in \cite{Brumfiel-(1970))}. Similarly as in that case, we obtain a splitting as a direct sum of a free abelian group and a torsion group, see Corollary~\ref{cor:norm-inv-splitting}, but in our case we are able to determine the free part for any $n,k\in\N$ as the subgroup of integral solutions of certain systems of linear equations; see Lemma~\ref{generkerJlema}. These systems can be solved algorithmically, we compute specific values of a particular choice of generators  for some low $k$ and $n\leq18$; see Tables~\ref{kerJgenertable},~\ref{kerJgenertable-l-big-n-4}. Brumfiel uses a different approach via Adams operations and only states results without computation for $\cp^n$, $n\leq6$. 

Recall the space $G/O$ defined as the homotopy fiber of the map of classifying spaces $J:BO\to BG$ representing the spherical fibration associated to the universal bundle over $BO$. This space is relevant since, for any compact smooth manifold $X$, we have an isomorphism $\sN_{\partial}^{DIFF}(X)\cong[X/\partial X,G/O]$. Since 
\[
\textstyle\cpd/\cp^n\times S^{k-1} = Th(\cpd)\simeq\Sigma^k\cp^n_+\simeq\Sigma^k\cp^n\vee S^k,
\]
the normal invariants split as 
\begin{align}
\label{norminvsplitting}
    \sN_{\partial}^{DIFF}(\cpd)\cong[\Sigma^k\cp^n,G/O]\oplus\pi_k(G/O).    
\end{align}
The homotopy groups of $G/O$ are known in low dimensions \cite[Remark 9.22]{ranicki2002algebraic}. The long exact sequence obtained by applying $[\Sigma^k\cp^n,\blank]$ to the fibration sequence associated to $J$ gives
\begin{align}
\label{sesnormalinv}
    \textstyle0\to\coker[\Sigma^k\cp^n,\Omega J]\xrightarrow{}[\Sigma^k\cp^n,G/O]\xrightarrow{[\Sigma^k\cp^n,r]}\ker[\Sigma^k\cp^n,J]\to 0
\end{align}
where $r \colon G/O \rightarrow BO$ is the standard inclusion.
\begin{lema}
\label{suspensg/o}                    
    There is a splitting of the sequence \eqref{sesnormalinv}:
    \[
    \textstyle[\Sigma^{k}\cp^n,G/O]\cong\coker[\Sigma^{k}\cp^n,\Omega J]\oplus\ker[\Sigma^{k}\cp^n,J]
    \]
    where $\ker[\Sigma^{k}\cp^n,J]\cong\Z^{t_{n,k}^{'}}$ with
    \begin{align*}
        t_{n,k}^{'}=
        \begin{cases}
            \lfloor\frac{n}{2}\rfloor &\text{if }k\equiv0\mod4\text{ or }k\equiv2\mod4\text{ and }n\text{ is even} \\
            \lfloor\frac{n}{2}\rfloor+1 &\text{if }k\equiv2\mod4\text{ and }n\text{ is odd} \\
            0 & \text{if }k\text{ is odd}
        \end{cases}
    \end{align*}
    Furthermore, $\coker[\Sigma^{k}\cp^n,\Omega J]$ is the torsion subgroup, and
    \begin{align*}
        \textstyle\coker[\Sigma^{2l}\cp^n,\Omega J]\cong
        \begin{cases}
            [\Sigma^{2l}\cp^n,G]/\Z_2 & \text{if }(l,n)\equiv(1,3)\text{ or } (3,1)\mod4 \\
            [\Sigma^{2l}\cp^n,G] & \text{otherwise}
        \end{cases}
    \end{align*}
\end{lema}
\noindent\textbf{Proof: }From Lemma \ref{KOCP} in the Appendix we see that $\widetilde{KO}^{-2l}(\cp^n)\cong\Z^{t_{n,2l}^{'}}$ up to elements of order two:
\[
    g_{\R}^{s}\cdot\mu_0^{t_{4m+1,8s}^{'}+1}\in\widetilde{KO}^{-8s}(\cp^{4m+1})\text{ and } g_{\R}^{s}\cdot\mu_2\mu_0^{t_{4m+3,8s+4}^{'}}\in\widetilde{KO}^{-8s-4}(\cp^{4m+3}).
\]

Suppose for now that these elements of order two have non-zero image under $[\Sigma^{2l}\cp^n,J]$. Since $[\Sigma^{2l}\cp^n,BG]$ is finite (this follows from the finitness of $\pi_i^s$, $i>0$ using the Atiyah-Hirzebruch spectral sequence), we get $\ker[\Sigma^{2l}\cp^n,J]\cong\Z^{t_{n,2l}^{'}}$. That proves the second isomorphism. The short exact sequence (\ref{sesnormalinv}) then splits, which gives the first isomorphism.

Now, again by Lemma~\ref{KOCP}, $\widetilde{KO}^{-2l-1}(\cp^n)=0$ for $(l,n)\not\equiv(1,3),(3,1)\mod4$. It follows that in these cases $\coker[\Sigma^{2l}\cp^n,\Omega J]\cong[\Sigma^{2l}\cp^n,G]$. In the other cases, $\widetilde{KO}^{-2l-1}(\cp^n)\cong\Z_2$. Denote by $\alpha_{n,l}$ the generator of this group. We now show that $[\Sigma^{2l}\cp^n,\Omega J](\alpha_{n,l})\neq0$. First, we observe that $[\Sigma^{2l}\cp^n,\Omega J]$ is the map $$\textstyle\widetilde{KO}^{-2l-1}(\cp^n)\xrightarrow{[\Sigma^{2l+1}\cp^n,J]}[\Sigma^{2l+1}\cp^n,BG].$$ Secondly, we apply $[\blank,BO]$ to the cofibration sequence 
\[
\textstyle\Sigma^{2l}S^{2n+2}\xrightarrow{\Sigma^{2l}H_{n}}\Sigma^{2l+1}\cp^{n}\xrightarrow{\Sigma^{2l+1}i_{n+1,n}}\Sigma^{2l+1}\cp^{n+1}\xrightarrow{\Sigma^{2l+1}j_n}\Sigma^{2l+1}S^{2n+2}
\]
where $H_n$ is the attaching map of the top cell of $\cp^{n+1}$, $i_{n+1,n}:\cp^n\hookrightarrow\cp^{n+1}$ is the standard embedding, and $j_n:\cp^{n+1}\to S^{2n+2}$ collapses the $(2n+1)$-skeleton. We get for $(l,n)\equiv(1,3)$ or $(3,1)\mod4$: 
\[
    \widetilde{KO}^{-2l-1}(\cp^n)\xrightarrow[\cong]{\widetilde{KO}^0(\Sigma^{2l}H_n)}\widetilde{KO}^{0}(S^{2l+2n+2})\cong\Z_2.
\]
Now, the Adams conjecture, proved in \cite{quillen1971adams}, implies that after localizing the map $J$ at $2$, $\ker[\Sigma^{2l+1}\cp^n,J_{(2)}]$ is generated by elements 
\[
\{\Psi_{\R}^b(\alpha_{n,l})-\alpha_{n,l}\}, b\in\Z, b\text{ odd}
\]
(this follows from \cite[Proposition 3.1]{Adams} and \cite[Theorem 1.1]{AdamsIII}) where $\Psi_{\R}^b$ is the $b$-th Adams operation \cite[Section 5]{Adams_vect}. The same holds for the kernel of $[S^{2l+2n+2},J_{(2)}]$. Since the Adams operations are natural, we have: 
\[
\ker[\Sigma^{2l+1}\cp^n,J_{(2)}]\xrightarrow[\cong]{\widetilde{KO}^0(\Sigma^{2l}H_n)}\ker[S^{2l+2n+2},J_{(2)}].
\]
But $\ker[S^{2l+2n+2},J]=0$ by \cite[Theorem 1.3]{adamsIV}, so $[\Sigma^{2l+1}\cp^n,J]$ is a monomorphism and $$\textstyle\coker[\Sigma^{2l}\cp^n,\Omega J]\cong[\Sigma^{2l}\cp^n,G]/\Z_2.$$ Moreover, $[\Sigma^{2l}\cp^n,G]$ is finite, which proves that $\coker[\Sigma^{k}\cp^n,\Omega J]$ is the torsion subgroup.

It remains to show that the elements $g_{\R}^{s}\cdot\mu_0^{t_{4m+1,8s}^{'}+1}$ and $g_{\R}^{s}\cdot\mu_2\mu_0^{t_{4m+3,8s+4}^{'}}$ from the first paragraph have non-zero images under $[\Sigma^{2l}\cp^n,J]$. The proof starts by noticing that $\widetilde{KO}^{-8s}(\cp^{4m+1}/\cp^{4m})\cong\widetilde{KO}^0(S^{8(s+m)+2})\cong\Z_2$ is generated by $g_{\R}^{s}\cdot\mu_0^{t_{4m+1,8s}^{'}+1}$ (and similarly for $g_{\R}^{s}\cdot\mu_2\mu_0^{t_{4m+3,8s+4}^{'}}$). The rest follows analogously to the proof of the isomorphism $\im[\Sigma^{2l}\cp^n,\Omega J]\cong\Z_2$ above.   \qed
\begin{cor}
\label{cor:norm-inv-splitting}
    There is a splitting 
    \begin{align}
        \sN_{\partial}^{DIFF}(\cpd)\cong\Z^{t_{n,k}^{'}+\epsilon_k}\oplus T_{n,k}^{'}
    \end{align}
    where $T_{n,k}^{'}$ is some finite abelian group, and $\epsilon_k=1$ if $k\equiv0\mod4$ and zero otherwise.
\end{cor}
\noindent\textbf{Proof: }The homotopy groups $\pi_k(G/O)$ are isomorphic to a direct sum of $\Z$ with a finite abelian group if $k\equiv0\mod4$ and are finite abelian otherwise. This follows from the long exact sequence of the homotopy groups of the fibration associated to $J$. This knowledge, together with the splitting \eqref{norminvsplitting} and Lemma~\ref{suspensg/o}, gives us the statement. \qed 
\begin{com}
\label{com:relation-exponents}
    Observe that the relation between the exponent $t_{n,k}$ from Theorem~\ref{mainthm1} and $t_{n,k}^{'}$ from Lemma~\ref{suspensg/o} is by their definition:
    \begin{align}
    \label{eqn:relation-between-exponents}
        t_{n,k}=
        \begin{cases}
            t_{n,k}^{'}+\epsilon_k-1 & \text{if }2n+k\equiv0\mod4 \\
            t_{n,k}^{'}+\epsilon_k & \text{otherwise}
        \end{cases}
    \end{align}
\end{com}
\begin{com}
\label{homotgroupG/O}
    The surgery obstruction map is a homomorphism with target $\Z_2$ if $2n+k\equiv2\mod4$. Since by (\ref{SESeven}) the structure set $\sS_{\partial}^{DIFF}(\cpdt)\cong\ker\sigma_{n,2l}^{DIFF}$, these two facts give us together with Corollary~\ref{cor:norm-inv-splitting} the statement of Theorem \ref{mainthm1} in those dimensions.
\end{com}
In order to compute the surgery obstructions, the results of Lemma \ref{suspensg/o} are not sufficient. We now present how to obtain a specific description of elements in $\ker[\Sigma^{k}\cp^n,J]$ as stable bundles in $\widetilde{KO}^0(\Sigma^k\cp^n)$. 

Recall that the Pontryagin character $ph:\widetilde{KO}^*(X)\to \widetilde{H}^*(X;\Q)$ is given by the composition 
\[
\widetilde{KO}^*(X)\xrightarrow{c}\widetilde{K}^*(X)\xrightarrow{ch}\widetilde{H}^*(X;\Q)
\]
of the complexification map $c$ with the Chern character.
Let $$\textstyle sh:\widetilde{KO}^0(\Sigma^k\cp^n)\to1+\displaystyle\prod_{s>0}\textstyle\widetilde{H}^{4s}(\Sigma^k\cp^n;\Q)$$ denote the characteristic class, whose universal class in $H^*(BO;\Q)$ corresponds (as a multiplicative sequence) to the formal power series $$\dfrac{e^{\frac{x}{2}}-e^{\frac{-x}{2}}}{x}$$For more details, consult \cite[\S5]{Adams}. Define coefficients $\alpha_{2s}$ by $$\log\left(\dfrac{e^{\frac{x}{2}}-e^{\frac{-x}{2}}}{x}\right)=\sum_{s=1}^\infty\alpha_{2s}\dfrac{x^{2s}}{(2s)!}$$ There is an identity obtained by Adams relating $sh$ to the Pontryagin character, which will be useful in our calculations:
\begin{lema}\cite[5.2]{Adams}
    \label{shlog}
    For any space $X$ and $x\in\prod_{i>0}\widetilde{H}^{4i}(X;\Q)$ define $\log(1+x)$ by means of the usual power series expansion. Then $$\log sh(\xi)=\sum_{s=1}^\infty\frac{\alpha_{2s}}{2}ph_{2s}(\xi)$$ where $ph_{t}$ is the summand of the Pontryagin character belonging to $\widetilde{H}^{2t}(X;\Q)$.
\end{lema}
The image of $[X,BO]\xrightarrow{[X,J]}[X,BG]$ was studied by Adams in \cite{AdamsI},\cite{Adams},\cite{AdamsIII} and \cite{adamsIV}, resp. by Adams and Walker for $X=\cp^n$ in \cite{Adams_Walker_1965}. In the latter, $V(\cp^n)$ was defined as the subgroup of elements $\xi\in\widetilde{KO}^0(\cp^n)$ for which there exists some $\beta\in\widetilde{KO}^0(\cp^n)$ such that $$sh(\xi)=ph(1+\beta),$$ and it was proved to be isomorphic to $\ker[\cp^n,J]$ unless $n=4t+1,$ $t>0$ when $\ker[\cp^n,J]$ is a subgroup of index 2 in $V(\cp^n)$.
We use the results of \cite[Section~2]{federgitler} to show that this strategy can be applied to the case where $X=\Sigma^k\cp^n$. For the purpose of proving the next lemma, we define 
\begin{align*}
\textstyle J'(\Sigma^k\cp^n)=\widetilde{KO}^0(\Sigma^k\cp^n)/V(\Sigma^k\cp^n), \\ 
\hspace{1pt}J(\Sigma^k\cp^n)=\widetilde{KO}^0(\Sigma^k\cp^n)/\ker[\Sigma^k\cp^n,J].
\end{align*}
\begin{lema}
\label{kerJeven}
    If $(l,n)\not\equiv(0,1)$ or $(2,3)\mod4$, then $\ker[\Sigma^{2l}\cp^n,J]=V(\Sigma^{2l}\cp^n)$.
\end{lema}
\noindent\textbf{Proof: }In these cases $\widetilde{KO}^0(\Sigma^{2l}\cp^n)$ is torsion-free, so it follows from \cite[Corollary 2.7]{federgitler} that $J(\Sigma^{2l}\cp^n)\cong J'(\Sigma^{2l}\cp^n)$. The isomorphism is given by the second map in the factorization of the quotient map $q':\widetilde{KO}^0(\Sigma^{2l}\cp^n)\twoheadrightarrow J'(\Sigma^{2l}\cp^n)$:
\[
    \begin{tikzcd}
        \widetilde{KO}^0(\Sigma^{2l}\cp^n) \arrow["q",twoheadrightarrow]{r} \arrow["q'",twoheadrightarrow,bend right=15]{rr} & J(\Sigma^{2l}\cp^n) \arrow{r}{\theta'} & J'(\Sigma^{2l}\cp^n)
    \end{tikzcd}
    \qed
\]
\begin{com}
\label{kerJunderembedding}
    In the remaining cases, $J(\Sigma^{2l}\cp^n)\cong J'(\Sigma^{2l}\cp^n)\oplus\Z_2$. This is obvious since, from the proof of Lemma \ref{suspensg/o}, we have elements of order $2$ 
    $$g_{\R}^{s}\mu_{0}\mu_i^{t_{4m+j,8s+2i}^{'}}\notin \ker[\Sigma^{8s+2i}\cp^{4m+j},J],$$ but $$sh(g_{\R}^{s}\mu_{0}\mu_i^{t_{4m+j,8s+2i}^{'}})=0\hspace{5pt}\Rightarrow\hspace{5pt}g_{\R}^{s}\mu_{0}\mu_i^{t_{4m+j,8s+2i}^{'}}\in V(\Sigma^{8s+2i}\cp^{4m+j})$$ for $(i,j)=(0,1)$ or $(2,3)$. However, the specific elements in $\ker[\Sigma^{2l}\cp^n,J]$ can be obtained by restriction $\widetilde{KO}^0(i_{n+1,n})$ from $\ker[\Sigma^{2l}\cp^{n+1},J]$, as follows from the following diagram
    \[
    \begin{tikzcd}
        \widetilde{KO}^0(\Sigma^{2l}(\cp^{n+1}/\cp^{n})) \arrow{r}{\widetilde{KO}^0(j)} \arrow[twoheadrightarrow]{d} & \widetilde{KO}^0(\Sigma^{2l}\cp^{n+1}) \arrow[column sep=3cm]{r}{\widetilde{KO}^0(i_{n+1,n})} \arrow[twoheadrightarrow]{d} & [0.5cm] \widetilde{KO}^0(\Sigma^{2l}\cp^{n}) \arrow{r}{} \arrow[twoheadrightarrow]{d} & 0 \\ 
        J(\Sigma^{2l}(\cp^{n+1}/\cp^{n})) \arrow{r}{\widetilde{KO}^0(j)|_{J(\dots)}} & J(\Sigma^{2l}\cp^{n+1}) \arrow{r}{\widetilde{KO}^0(i_{n+1,n})|_{J(\dots)}} & J(\Sigma^{2k}\cp^{n}) \arrow{r}{} & 0 
    \end{tikzcd}
    \]
    obtained from \cite[3.12]{Adams} and \cite[1.1]{AdamsIII}.
\end{com}

So, for $(l,n)$ as in Lemma \ref{kerJeven}, the virtual (stable) bundle $\xi$ belongs to the subgroup $\ker[\Sigma^{2l}\cp^n,J]$ if and only if $sh(\xi)=ph(1+\beta)$ for some $\beta\in\widetilde{KO}^0(\Sigma^{2l}\cp^n)$. Since $\log$ is monomorphic, by using Lemma \ref{shlog}, this equation is equivalent to 
\begin{align}
\label{equationkerJ}
    \sum_{s=1}^\infty\frac{\alpha_{2s}}{2}ph_{2s}(\xi)=\log ph(1+\beta)=\sum_{i=1}^\infty(-1)^{i-1}\frac{ph(\beta)^i}{i}. 
\end{align}

We introduce the formal series 
\[
    h(x)=e^x-2+e^{-x}, \hspace{5pt} g_i(x)=(e^x-e^{-x})h(x)^i.
\]
\begin{lema}
\label{cherngener}
    Unless the generators of $\widetilde{KO}^0(\Sigma^{2l}\cp^{n})$ are of order two, their Pontryagin characters are given by:
    \begin{align*}
        ph(g_{\R}^s\cdot\mu_0^i)&=z^{\times8s}\times h(x)^i \\
        ph(g_{\R}^s\cdot\mu_2\mu_0^{i-1})&=z^{\times(8s+4)}\times h(x)^i \\
        ph(g_{\R}^s\cdot\mu_1\mu_0^{i-1})&=z^{\times(8s+2)}\times g_{i-1}(x) \\
        ph(g_{\R}^s\cdot\mu_3\mu_0^{i-1})&=z^{\times(8s+6)}\times g_{i-1}(x),\hspace{15pt} i\geq1 \\
        ph(g_{\R}^s\cdot\sigma)&=\frac{1}{2}ph(g_{\R}^s\cdot\mu_1\mu_0^{\lfloor\frac{n}{2}\rfloor}),\hspace{5pt} ph(g_{\R}^s\cdot\tau)=\frac{1}{2}ph(g_{\R}^s\cdot\mu_3\mu_0^{\lfloor\frac{n}{2}\rfloor})
    \end{align*}
    where $s=\lfloor\frac{l}{4}\rfloor$, $x$ generates  $H^2(\cp^{n};\Z)$, and the cross product with $z\in H^1(S^1;\Z)$ gives the suspension isomorphism $\widetilde{H}^i(X;\Z)\xrightarrow{\cong}\widetilde{H}^{i+1}(\Sigma X;\Z)$. 
\end{lema}
\noindent\textbf{Proof: }The proof will be presented for the case of $g_{\R}^s\cdot\mu_1\mu_0^{i-1}$. The calculation is analogous in the remaining cases.
Since $c\circ r=\id+*:\widetilde{K}^*(X)\to\widetilde{K}^*(X)$ \cite[3.9]{Adams_vect}, and the Chern character is given on complex line bundles $L$ by $e^{c_1(L)}$, we get from \cite[Theorem 3.25.a)]{karoubi} and the definition of $\mu_1$ and $\mu_0$ in Appendix that:
\begin{align*}
    ph(g_{\R}^s\cdot\mu_1\mu_0^{i-1})&=[ch(c(g_{\R}))]^sch(c(\mu_1))[ch(c(\mu_0))]^{i-1} \\
    &=[ch(c(g_{\R}))]^sch(g\cdot(H-H^{-1}))[ch(H-2+H^{-1})]^{i-1} \\
    &=[\underbrace{z\times\dots\times z}_\text{$8s$ times}]\times [z^{\times2}\times(e^x-e^{-x})(e^x-2+e^{-x})^{i-1}] \qed
\end{align*}

Denote by $[h(x)^j]_i$ the coefficient of $x^i$ in the power series of $h(x)^j$ (and similarly for $g_j(x)$). Define:
\[
    a_i=\frac{2}{\alpha_{2\left(\lfloor\frac{l}{2}\rfloor+i\right)}}\text{ and }r_l=
    \begin{cases}
        \lceil\frac{n}{2}\rceil & \text{if $l$ is odd} \\
        \lfloor\frac{n}{2}\rfloor & \text{if $l$ is even}
    \end{cases}
\]
We now prove that for $l>0$ the equation~\eqref{equationkerJ} leads to a linear system over $\Q$ described by one of the following $(r_l\times2r_l)$-matrices:
\begin{align}
    \label{eqn:kerJ-matrix1}
    &
    \begin{pNiceMatrix}[nullify-dots, columns-width=6pt, xdots/inter=10pt]
        \frac{1}{a_1}                         & 0      & & \Cdots[shorten=15pt] & 0 
        & -1                              & 0      & & \Cdots[shorten=15pt] & 0 \\
        \frac{1}{a_2}h(x)_4                   & \Ddots[shorten=20pt] & & & \Vdots[shorten=5pt] 
        & -h(x)_4                         & \Ddots[shorten=18pt] & & & \Vdots[shorten=5pt] \\
        \Vdots                                &        & &        &  
        & \Vdots                          &        & &        &  \\
                                          &        & &        & 0 
        &                                 &        & &        & 0 \\
        \frac{1}{a_{r_l}}h(x)_{2r_l}                & \Cdots[shorten=5pt] & & \frac{1}{a_{r_l}}[h(x)^{r_l-1}]_{2r_l} & \frac{1}{a_{r_l}} 
        & -h(x)_{2r_l}                      & \Cdots[shorten=5pt] & & -[h(x)^{r_l-1}]_{2r_l} & -1
    \end{pNiceMatrix} 
\end{align}

    if $l$ is even, with an additional column of zeros at the end of both blocks if $(l,n)\equiv(0,1)$ or $(2,3)\mod4$,
\begin{align}
    \label{eqn:kerJ-matrix2}
    &
    \begin{pNiceMatrix}[nullify-dots, columns-width=6pt, xdots/inter=10pt]
        \frac{2}{a_1}                         & 0      & & \Cdots[shorten=15pt] & 0 
        & -2                              & 0      & & \Cdots[shorten=15pt] & 0 \\
        \frac{1}{a_2}g_0(x)_4                   & \Ddots[shorten=20pt] & & & \Vdots[shorten=5pt] 
        & -g_0(x)_4                         & \Ddots[shorten=18pt] & & & \Vdots[shorten=5pt] \\
        \Vdots                                &        & &        &  
        & \Vdots                          &        & &        &  \\
        &  & & & 0 
        &    & & &  & 0 \\
        \frac{1}{a_{r_l}}g_0(x)_{2r_l}                & \Cdots[shorten=5pt] & & \frac{1}{a_{r_l}}g_{r_l-2}(x)_{2r_l} & \frac{2}{a_{r_l}} 
        & -g_0(x)_{2r_l}                      & \Cdots[shorten=5pt] & & -g_{r_l-2}(x)_{2r_l} & -2
    \end{pNiceMatrix} \\ \nonumber
    &\text{if }l\text{ is odd and $(l,n)\not\equiv(3,1)$ or $(1,3)\mod4$},\text{ or}
\end{align}
\begin{align}
\label{eqn:kerJ-matrix3}
    &
    \begin{pNiceMatrix}[nullify-dots, columns-width=3pt, xdots/inter=10pt]
        \frac{2}{a_1}                            & 0      & & \Cdots[shorten=15pt] & 0 
        & -2                                 & 0      & & \Cdots[shorten=15pt] & 0 \\
        \frac{1}{a_2}g_0(x)_4                      & \Ddots[shorten=20pt] & &        & \Vdots[shorten=5pt] 
        & -g_0(x)_4                            & \Ddots[shorten=5pt] & &        & \Vdots[shorten=5pt] \\
        \Vdots[inter=3.5pt]                                   &        & &        &  
        & \Vdots[inter=5pt]                             &        & &        &  \\
        \frac{1}{a_{r_l-1}}g_0(x)_{2r_l-2}           & \Cdots[shorten=5pt] & & \frac{2}{a_{r_l-1}} & 0 
        & -g_0(x)_{2r_l-2}                     & \Cdots[shorten=5pt] & & -2     & 0 \\
        \frac{1}{a_{r_l}}g_0(x)_{2r_l}                   & \Cdots[shorten=5pt] & & \frac{1}{a_{r_l}}g_{r_l-2}(x)_{2r_l} & \frac{1}{a_{r_l}} 
        & -g_0(x)_{2r_l}                         & \Cdots[shorten=5pt] & & -g_{r_l-2}(x)_{2r_l} & -1
    \end{pNiceMatrix} \\
    \nonumber
    &\text{if $(l,n)\equiv(3,1)$ or $(1,3)\mod4$}.
\end{align}
\begin{lema}
\label{generkerJlema}
    For $l>0$, the subgroup $\ker[\Sigma^{2l}\cp^n,J]$ is given by restrictions to the first $r_l$ coordinates of integer solutions of the system corresponding to: 
    \begin{itemize}
        \item the matrix~\eqref{eqn:kerJ-matrix1} if $l$ is even and $(l,n)\not\equiv(0,1),(2,3)\mod4$,
        \item the matrix~\eqref{eqn:kerJ-matrix2} if $l$ is odd and $(l,n)\not\equiv(3,1),(1,3)\mod4$,
        \item the matrix~\eqref{eqn:kerJ-matrix3} if $(l,n)\equiv(3,1),(1,3)\mod4$.
    \end{itemize}
    
    For $0\leq l\leq3$ and $n=17$ or $18$, its generators are summarized in Table \ref{kerJgenertable}. 

    For $4\leq l\leq15$ and $n=4$ its generators are summarized in Table \ref{kerJgenertable-l-big-n-4}.
\end{lema}

\noindent\textbf{Proof: }We describe the proof for $l\equiv2\mod4$, $n$ even. The remaining cases follow analogously.
From Lemma~\ref{kerJeven} we see that we need to compute $V(\Sigma^{2l}\cp^n)$. We express $\xi,\beta\in\widetilde{KO}^0(\Sigma^{2l}\cp^n)$ as $$\xi=g_{\R}^t\cdot(x_1\mu_2+x_2\mu_2\mu_0+\dots+x_{r}\mu_2\mu_0^{r-1}),\hspace{5pt} \beta=g_{\R}^t\cdot(y_1\mu_2+y_2\mu_2\mu_0+\dots+y_{r}\mu_2\mu_0^{r-1})$$ where $x_i,y_i\in\Z$, $r=\frac{n}{2}$ and $t=\lfloor\frac{l}{4}\rfloor$. The cup products in $H^*(\Sigma^{2l}\cp^n;\Q)$ vanish, so~\eqref{equationkerJ} simplifies to 
\begin{equation} \label{eqn:simplified-equationkerJ}
    \sum_{s=1}^\infty\frac{\alpha_{2s}}{2}ph_{2s}(\xi)=ph(\beta).    
\end{equation}
Substituting $\xi$ and $\beta$ into \eqref{eqn:simplified-equationkerJ} we obtain the following system of $r$ linear equations in $x_1,\dots,x_r,y_1,\dots,y_r$ over $\Q$:
\begin{align*}
    \frac{\alpha_{2s}}{2}(x_1ph(\mu_2)_{2s}+\dots+x_rph(\mu_2\mu_0^{r-1})_{2s})&=y_1ph(\mu_2)_{2s}+\dots+y_rph(\mu_2\mu_0^{r-1})_{2s}, \\ 
    \lfloor\textstyle\frac{l}{2}\rfloor+1\leq &s\leq\lfloor\textstyle\frac{l}{2}\rfloor+r    
\end{align*}
After substituting the values of Pontryagin characters from Lemma~\ref{cherngener}, this equation gives us the system described by the matrix~\eqref{eqn:kerJ-matrix1}. The first $r$ coordinates $(x_1,...,x_r)$ of the integral solutions of ~\eqref{eqn:kerJ-matrix1} thus describe $\xi\in\ker[\Sigma^{2l}\cp^n,J]$.

The numerical values of the entries in the matrix, as well as the integral solutions of this system for $l$ and $n$ in the range mentioned in the lemma, were obtained algorithmically using the Python module SymPy.\qed
\begin{com}
    The cases where $(l,n)\equiv(0,1)$ or $(2,3)\mod4$ were omitted from the lemma. In these cases, the integral solutions of the system~\eqref{eqn:kerJ-matrix1} (with additional columns of zeros, as mentioned below the matrix) give us the subgroup $V(\Sigma^{2l}\cp^n)$. However, $\ker[\Sigma^{2l}\cp^n,J]$ is then only some subgroup of index $2$ of $V(\Sigma^{2l}\cp^n)$ (Lemma~\ref{kerJeven} fails). Still, the generators can be computed for these cases as restrictions of generators for higher $n$, as described in Remark~\ref{kerJunderembedding}.     
\end{com}
\begin{com}
\label{rem:complexity-k-is-0-gener-ker-J}
    Table~\ref{kerJgenertable} includes the generators of $\ker[\cp^{18},J]$ (the case when $l=0$). In that case, the generators are given by integral solutions of the equation \eqref{equationkerJ} which is not linear after substituting the coordinate forms of $\xi,\beta$ (the non-vanishing cup products induce non-linearity on the right side of the equation). Still, we were able to solve this polynomial system using SymPy, although the problem is computationally much more complex.    
\end{com}

\begin{com}[Coordinates of generators]
    The coordinates of any set of generators of integral solutions of the systems~\eqref{eqn:kerJ-matrix1}, \eqref{eqn:kerJ-matrix2} and \eqref{eqn:kerJ-matrix3} are of the form
    \begin{align*}
        &\left(0,...,\num (a_r),0,...,\denom(a_r)\right),...,\left(\num (a_1),b_{r,2},...,b_{r,r},\denom(a_1),c_{r,2},...,c_{r,r}\right)     
    \end{align*}
    where $b_{i,j},c_{i,j}\in\Z$. This follows inductively from the diagram 
    \begin{equation*}
        \begin{tikzcd}
            \widetilde{KO}^0(\Sigma^{2l}(\cp^{2m}/\cp^{2m-2})) \arrow{r}{\widetilde{KO}^0(j)} \arrow[twoheadrightarrow]{d} & \widetilde{KO}^0(\Sigma^{2l}\cp^{2m}) \arrow{r}{\widetilde{KO}^0(i)} \arrow[twoheadrightarrow]{d} & \widetilde{KO}^0(\Sigma^{2l}\cp^{2m-2}) \arrow{r}{} \arrow[twoheadrightarrow]{d} & 0 \\ 
            J(\Sigma^{2l}(\cp^{2m}/\cp^{2m-2})) \arrow{r}{\widetilde{KO}^0(j)|_{J(\dots)}} & J(\Sigma^{2l}\cp^{2m}) \arrow{r}{\widetilde{KO}^0(i)|_{J(\dots)}} & J(\Sigma^{2l}\cp^{2m-2}) \arrow{r}{} & 0 
        \end{tikzcd}
    \end{equation*}
    since $J(\Sigma^{2l}(\cp^{2m}/\cp^{2m-2}))\cong\Z_{\num(a_m)}$, which can be proved in a way analogous to \cite[4.10]{Adams_Walker_1965}. 
\end{com}
\begin{com}[Choice of generators]
\label{kerJgenerstandembed}
    There are many possible choices of generators of $\ker[\Sigma^{2l}\cp^n,J]$. A specific choice provides us with an isomorphism $\ker[\Sigma^{2l}\cp^n,J]\cong\Z^{t_{n,2l}^{'}}$. Together with a choice of some splitting in Lemma~\ref{suspensg/o} this gives the following:
    \[
        \Z^{t_{n,2l}^{'}}\xrightarrow{\cong}\ker[\Sigma^{2l}\cp^n,J]\hookrightarrow[\Sigma^{2l}\cp^n,G/O]
    \]
    For $m\leq n$, denote by $\Sigma^{2l}i_{m,n}:\Sigma^{2l}\cp^m\to\Sigma^{2l}\cp^n$ the suspension of the standard embedding. The map $[\Sigma^{2l}i_{m,n},BO]$ maps $\ker[\Sigma^{2l}\cp^n,J]$ onto $\ker[\Sigma^{2l}\cp^m,J]$. This follows from arguments similar to the one from Remark~\ref{kerJunderembedding}. In Section~\ref{nonsmooth} it will be important to make the choices of generators in the following way. 
    
    Suppose we have chosen the generators $\xi_{1},\dots,\xi_{t_{n-1,2l}^{'}}\in\ker[\Sigma^{2l},\cp^{n-1}]$. The first $t_{n-1,2l}^{'}$ generators of $\ker[\Sigma^{2l}\cp^n,J]$ are given by choosing an element in the preimage of $\xi_j$ under $[\Sigma^{2l}i_{n-1,n},BO]$ for $1\leq j\leq t_{n-1,2l}^{'}$. The remaining generator (when $t_{n-1,2l}^{'}<t_{n,2l}^{'}$) should map to zero under $[\Sigma^{2l}i_{n-1,n},BO]$. In other words, the choice should be such that the map $J_{m,n}$ in
    \[
        \begin{tikzcd}
            \Z^{t_{n,2l}^{'}} \arrow{r}{\cong} \arrow[dashed,twoheadrightarrow]{d}{J_{n,m}} & {\ker[\Sigma^{2l}\cp^n,J]} \arrow{r}{} \arrow[twoheadrightarrow]{d}{{[\Sigma^{2l}i_{m,n},BO]}} & {[\Sigma^{2l}\cp^n,G/O]} \arrow{d}{{[\Sigma^{2l}i_{m,n},G/O]}} \\
            \Z^{t_{m,2l}^{'}} \arrow{r}{\cong} & {\ker[\Sigma^{2l}\cp^m,J]} \arrow{r}{} & {[\Sigma^{2l}\cp^m,G/O]} 
        \end{tikzcd}
    \]
    is given by the $(t_{m,2l}^{'}\times t_{n,2l}^{'})$-matrix
    \[
    \begin{pNiceMatrix}[nullify-dots, columns-width=6pt]
        1      & 0      & & \Cdots[shorten=15pt] & 0 
        & 0      &       & & \Cdots[shorten=15pt] & 0 \\
        0      & \Ddots[shorten=20pt] & & & \Vdots[shorten=10pt] 
        &       & \Ddots[shorten=18pt] & & & \Vdots[shorten=10pt] \\
        \Vdots &        & &        &  
        & \Vdots &        & &        &  \\
           &        & &        & 0 
        &        &        & &        &  \\
        0      & \Cdots[shorten=5pt] & & 0 & 1 
        & 0      & \Cdots[shorten=5pt] & &  & 0
    \end{pNiceMatrix}
    \]
    
    In particular, for $0\leq l\leq3$ and $m<17$ or $18$, the generators can be obtained from Table \ref{kerJgenertable} as the images of generators for $n=17$ or $18$ under $\widetilde{KO}^0(\Sigma^{2l}i_{m,n})$.
\end{com}

%% file: obstruction.tex
\section{The Surgery Obstruction}
\label{obstrsect}
In this section, we determine the surgery obstruction maps 
\begin{align}
\label{surgobstrformulaeven}
    \textstyle\sigma_{n,2l}^{DIFF}:\sN_{\partial}^{DIFF}(\cpdt)&\to\Z \\ \nonumber
    [M,\partial M,f,\partial f]&\textstyle\mapsto\frac{1}{8}(\sign(M)-\sign(\cpdt)) 
\end{align}
for even $n+l$. In Section \ref{norminvsect} we computed the normal invariants in terms of Sullivan's isomorphism 
\begin{align}
\label{sullivaniso}
    \textstyle\sN_{\partial}^{DIFF}(\cpd)\cong[Th(\cpd),G/O]    
\end{align}
Elements in $[Th(\cpd),G/O]$ can be viewed as equivalence classes $[(\xi,t)]$, where $\xi$ is a vector bundle over $Th(\cpd)$ and $t:S(\xi)\to Th(\cpd)\times S^{N}$ is a trivialization of the associated spherical fibration. The isomorphism (\ref{sullivaniso}) sends this element to a normal bordism class represented by a normal map
\begin{equation}
\label{sect3normmap}
    \begin{tikzcd}
        \nu_{M}\oplus\underline{\R^b} \arrow{r}{\overline{f}} & \nu_{\cpd}\oplus\xi \\
        (M,\partial M) \arrow{u}{} \arrow{r}{(f,\partial f)} & (\cpd,\cp^n\times S^{k-1}) \arrow{u}{}
    \end{tikzcd}
\end{equation}
The construction of this normal map involves the Pontryagin-Thom construction, as well as a third alternate definition of the set of normal invariants. Note that the trivialization $t$ plays a role in this construction, it is possible for elements $[(\xi,t)]$ and $[(\xi,t')]$ to represent two different normal bordism classes. More details can be found in \cite[Chapter 7]{Lueck-Macko(2024)}.
\begin{lema}
\label{gluinglema}
    Gluing to $(M,\partial M,f,\partial f)$ a trivial normal map over $\cp^n\times D^k$ along their common boundary gives a map
    \begin{align*}
        \textstyle\text{ext}:\sN_{\partial}^{DIFF}(\cp^n\times D^k)&\textstyle\to\sN^{DIFF}(\cp^n\times S^k) \\
        [M,\partial M,f,\partial f]&\mapsto\textstyle[M\cup\cp^n\times D^k,f\cup\id]
    \end{align*}
    This map agrees with $[\pi,G/O]:[Th(\cpd),G/O]\to[\cp^n\times S^k,G/O]$ under the isomorphism \eqref{sullivaniso}, where $\pi:\cp^n\times S^k\to Th(\cpd)$ is the map that collapses $\cp^n\times\{\text{pt.}\}$ to a point.
\end{lema}
\noindent\textbf{Proof: } The map $f:M\to\cpd$ underlying the normal map $(M,\partial M,f,\partial f)$ restricts to a diffeomorphism on the boundary, we can therefore glue the normal map $(\cpd,\cp^n\times S^{k-1},\id,\id)$ along the boundary to obtain the normal map
\begin{equation}
\label{sect3normmapclosed}
    \begin{tikzcd}
        \nu_{M\cup\cpd}\oplus\underline{\R^a} \arrow{r}{\overline{f\cup\id}} \arrow{d}{} & \nu_{\cp^n\times S^{k}}\oplus(\xi\cup\underline{\R^c}) \arrow{d}{} \\
        M\cup_{\partial M}\cpd  \arrow{r}{f\cup\id} & \cp^n\times S^{k}
    \end{tikzcd}
\end{equation}
where $a,c\in\N$, $[\xi,t]\in[Th(\cpd),G/O]$ corresponds to $[M,\partial M,f,\partial f]$ under \eqref{sullivaniso}, so the resulting normal map $[M\cup\cpd,f\cup\id]$ corresponds to $[\xi\cup\underline{\R}^c,t']$ for some trivialization $$\textstyle t':S(\xi\cup\underline{\R}^c)\xrightarrow{\simeq}(\cp^n\times S^k)\times S^c.$$ This construction respects normal bordism classes, since we can glue a trivial normal bordism to the bordism between two representants $(M,\partial M,f,\partial f)$ and $(M',\partial M',f',\partial f')$. 

The map $\pi$ is homotopic to the map that collapses $\cpd_+$ to a point, the image of $[\pi,G/O]$ is therefore given by $[\xi\cup\underline{\R}^c,t']$. \qed
\begin{com}
\label{rem:splitting-cp-times-S-thom-space}
    With respect to the following splittings
    \begin{equation*}
        \begin{tikzcd}
            {[Th(\cpd),Y]} \arrow{r}{[\pi,Y]} & {[\cp^n\times S^k,Y]} \\
            {[\Sigma^k\cp^n,Y]}\oplus\pi_k(Y) \arrow[u,phantom,sloped,"\cong"] & {[\Sigma^k\cp^n,Y]\oplus\pi_k(Y)\oplus[\cp^n,Y]} \arrow[u,phantom,sloped,"\cong"]
        \end{tikzcd}
    \end{equation*}
    the map $[\pi,Y]$ is given by the inclusion of the first two summands for $Y=G/O,BO$. In particular, the map ext from Lemma \ref{gluinglema} is a monomorphism.
\end{com}
The result can be summarized by the following diagram:
\begin{equation}
\label{normalmapdiag}
    \begin{tikzcd}
        {[M,\partial M,f,\partial f]} \arrow[mapsto]{r} \arrow[mapsto,bend right=70]{ddd} \arrow[d,phantom,sloped,"\in"] & {[\xi,t]} \arrow[mapsto]{r} \arrow[mapsto,bend right=70]{ddd} \arrow[d,phantom,sloped,"\in"] & {[\xi]} \arrow[mapsto,bend left=85]{ddd} \arrow[d,phantom,sloped,"\in"] \\
        \sN_{\partial}^{DIFF}(\cpd) \arrow{d}{\text{ext}} \arrow{r}{\cong} & {[Th(\cpd),G/O]} \arrow{d}{[\pi,G/O]} \arrow{r}{[Th(\cdot),r]} & {[Th(\cpd),BO]} \arrow{d}{[\pi,BO]} \\
        \sN^{DIFF}(\cp^n\times S^k) \arrow{r}{\cong} & {[\cp^n\times S^k,G/O]} \arrow{r}{[\cp^n\times S^k,r]} & {[\cp^n\times S^k,BO]} \\
        {[M\cup\cpd,f\cup\id]} \arrow[mapsto]{r} \arrow[u,phantom,sloped,"\in"] & {[\xi\cup\underline{\R}^a,t']} \arrow[mapsto]{r} \arrow[u,phantom,sloped,"\in"] & {[\xi\cup\underline{\R}^a]} \arrow[u,phantom,sloped,"\in"]
    \end{tikzcd}
\end{equation}
\begin{lema} \label{lem:surgery-obstruction-signature-formula}
    For even $n+l$ the composition $$\textstyle[Th(\cp^n\times D^{2l}),G/O]\xrightarrow{\cong}\sN_{\partial}^{DIFF}(\cp^n\times D^{2l})\xrightarrow{\sigma_{n,2l}^{DIFF}}L_{2(n+l)}(\Z)\cong\Z$$is given by the formula 
    \begin{align}
    \label{eqn:obstr-formula-Hirz-simplified}
        \textstyle\sigma_{n,2l}^{DIFF}([\xi,t])=\frac{1}{8}\langle H^*(\pi)(\sL^{-1}(\xi)-1)H^*(\proj_1)
    \sL(\cp^n), [\cp^n\times S^{2l}]\rangle
    \end{align}
    where $\proj_1:\cp^n\times S^{2l}\to\cp^n$ is the projection on the first factor. In particular, the surgery obstruction map depends only on the image of  $$[Th(\cp^n\times D^{2l}),G/O]\xrightarrow{[Th(\cp^n\times D^{2l}),r]}[Th(\cp^n\times D^{2l}),BO]\xrightarrow{[\pi,BO]}[\cp^n\times S^{2l},BO].$$
\end{lema}
\noindent\textbf{Proof: }Consider the normal map $[M\cup\cpdt,f\cup\id]\in\sN_{\partial}^{DIFF}(\cpdt)$. By Lemma \ref{gluinglema} it corresponds to $[\pi,G/O]([\xi,t])\in[\cp^n\times S^{2l},G/O]$. The additivity of the surgery obstruction \cite[II.1.4]{browder} gives $$\sigma_{n,2l}^{DIFF}(f\cup\id)=\sigma_{n,2l}^{DIFF}(f,\partial f)+\sigma_{n,2l}^{DIFF}(\id,\id)=\sigma_{n,2l}^{DIFF}(f,\partial f)=\sigma_{n,2l}^{DIFF}([\xi,t]).$$ There is a slight abuse of notation here, we denote by $\sigma_{n,2l}^{DIFF}$ the surgery obstruction with source $\sN_{\partial}^{DIFF}(\cpdt)$, as well as the one with source $\sN^{DIFF}(\cp^n\times S^{2l})$.

In this way, we converted the original "relative boundary" problem to a problem over a closed manifold. A formula analogous to (\ref{surgobstrformulaeven}) holds for $\cp^n\times S^{2l}$. Using Hirzebruch's Signature Theorem we get 
\begin{align}
    \nonumber
    \sigma_{n,2l}^{DIFF}(f\cup\id)&=\textstyle\frac{1}{8}(\sign(M\cup\cpdt)-\sign(\cp^n\times S^{2l})) \\
\label{Hirzformula2}
    &=\textstyle\frac{1}{8}(\langle\sL(\tau_{M\cup\cpdt}),[M\cup\cpdt]\rangle \\ \nonumber
    &-\langle\sL(\tau_{\cp^n\times S^{2l}}),[\cp^n\times S^{2l}]\rangle)
\end{align}
From the normal map (\ref{sect3normmapclosed}), we see that 
\begin{align*}
    \textstyle\sL(\tau_{M\cup\cpdt})&=\sL^{-1}(\nu_{M\cup\cpdt}) \\
    &=\textstyle H^*(f\cup\id)\sL^{-1}(\nu_{\cp^n\times S^{2l}}\oplus(\xi\cup\underline{\R^c})) \\
    &=\textstyle H^*(f\cup\id)\sL^{-1}(\nu_{\cp^n\times S^{2l}})H^*(f\cup\id)\sL^{-1}(\xi\cup\underline{\R^c}) \\
    &=\textstyle H^*(f\cup\id)\sL(\tau_{\cp^n\times S^{2l}})H^*(f\cup\id)\sL^{-1}(\xi\cup\underline{\R^c}) \\
    &=\textstyle H^*(f\cup\id)(\sL(\tau_{\cp^n\times S^{2l}})\sL^{-1}(\xi\cup\underline{\R^c}))
\end{align*}
and therefore 
\begin{align*}
    \textstyle\sigma_{n,2l}^{DIFF}(f\cup\id)&=\textstyle\frac{1}{8}(\langle H^*(f\cup\id)(\sL(\tau_{\cp^n\times S^{2l}})\sL^{-1}(\xi\cup\underline{\R^c})),[M\cup\cpdt]\rangle- \\
    &-\textstyle\langle\sL(\tau_{\cp^n\times S^{2l}}), [\cp^n\times S^{2l}]\rangle) \\
    &=\textstyle\frac{1}{8}(\langle\sL(\tau_{\cp^n\times S^{2l}})\sL^{-1}(\xi\cup\underline{\R^c}),H_*(f\cup\id)[M\cup\cpdt]\rangle- \\
    &-\textstyle\langle\sL(\tau_{\cp^n\times S^{2l}}), [\cp^n\times S^{2l}]\rangle) \\
    &=\textstyle\frac{1}{8}\langle(\sL^{-1}(\xi\cup\underline{\R^c})-1)\sL(\tau_{\cp^n\times S^{2l}}), [\cp^n\times S^{2l}]\rangle
\end{align*}
Since $[\tau_{\cp^n\times S^{2l}}]\cong[\proj_1,BO]([\tau_{\cp^n}])\oplus[\proj_2,BO]([\tau_{S^{2l}}])$, the class $\sL(\tau_{S^{2l}})=1$, and $[\xi\cup\underline{\R^c}]=[\pi,BO]([\xi])$, we obtain the desired formula. \qed
\\

Next, we prepare some notation for Lemma~\ref{lema:obstr-formula-simplified}. Recall that by Corollary~\ref{cor:norm-inv-splitting}:
\begin{align}
\label{eqn:norm-inv-specific-iso}
    [Th(\cpdt),G/O]&\cong\ker[\Sigma^{2l}\cp^n,J]\oplus \coker[\Sigma^{2l}\cp^n,\Omega J]\oplus\pi_{2l}(G/O) \\ \nonumber
    &\cong\Z^{t_{n,2l}^{'}+\epsilon_{2l}}\oplus T_{n,2l}^{'}
\end{align}
where $\epsilon_{2l}=1$ if $l$ is even and zero otherwise. Here $\ker[\Sigma^{2l}\cp^n,J]$ is a subgroup of $\widetilde{KO}^0(\Sigma^{2l}\cp^n)$ and its isomorphism with $\Z^{t_{n,2l}^{'}}$ is given by choosing a specific set of generators $\xi_1,\dots,\xi_{t_{n,2l}^{'}}$ of solutions of systems~\eqref{eqn:kerJ-matrix1}, \eqref{eqn:kerJ-matrix2} and ~\eqref{eqn:kerJ-matrix3}; see Lemma~\ref{generkerJlema}. Denote by
\begin{align}
\label{eqn:any-choice-of-generators-ker-J}
    (m_{1,i},\dots,m_{t_{n,2l}^{'},i})
\end{align}
the coordinates of $\xi_i$ in $\widetilde{KO}^0(\Sigma^{2l}\cp^n)\cong\Z^{t_{n,2l}^{'}}\oplus\text{torsion}$. 

In cases where $l$ is even, the additional $\Z$-summand comes from $\pi_{2l}(G/O)$. We denote its generator by $\xi_{S^{2l}}$. Note that, as a consequence of the last statement in Lemma~\ref{lem:surgery-obstruction-signature-formula}, $\sigma_{n,l}^{DIFF}([\xi,t])$ is independent of the choice of a particular splitting in Lemma~\ref{suspensg/o}.

In the following result, the Bernoulli numbers $B_i$ play a prominent role. We use the definition of $B_i$ from \cite[Appendix B]{milnor1974characteristic}. We define for $1\leq i\leq t_{n,2l}^{'}$:
\begin{align}
\label{eqn:coeff-a_n,l}
    a_{n,l}^i&=\frac{1}{8}\sum_{j>0}^{\frac{n+l}{2}}(-1)^j2^{2j}(2^{2j-1}-1)\frac{|B_j|}{2j}\left(\sum_{k=1}^{t_{n,2l}^{'}} m_{k,i}(z^{\times2l}\times h(x)^k)_{2j}\right)\times \\ \nonumber
    &\times\left[\left(\frac{x}{tanh(x)}\right)^{n+1}\right]_{2(\frac{n+l}{2}-j)} \\
\label{eqn:coeff-b_n,l}
    b_{n,l}^i&=\frac{1}{8}\sum_{j>0}^{\frac{n+l}{2}}(-1)^j2^{2j}(2^{2j-1}-1)\frac{|B_j|}{2j}\times \\ \nonumber
    &\times\left(\sum_{k=1}^{t_{n,2l}^{'}}m_{k,i} \gamma_{n,l}^{k}(z^{\times2l}\times g_{k-1}(x))_{2j}\right)\left[\left(\frac{x}{tanh(x)}\right)^{n+1}\right]_{2(\frac{n+l}{2}-j)}
\end{align}
where 
\[
    \gamma_{n,l}^k=
    \begin{cases}
        \frac{1}{2} & \text{if }k=t_{n,2l}^{'}\text{ and }(l,n)\equiv(1,3),(3,1)\mod4, \\
        1 & \text{otherwise},
    \end{cases}
\]
and
\begin{align}
\label{eqn:coeff-a_n,l-sphere}
    a_{n,l}^{0}=\frac{1}{8}(-1)^{\frac{l}{2}}\frac{3-(-1)^l}{2}2^{l-2}(2^{l-1}-1)\num\left(\frac{B_{\frac{l}{2}}}{2l}\right)\left[\left(\frac{x}{tanh(x)}\right)^{n+1}\right]_{n}.
\end{align}
Here, $f(x)_{2i}$ denotes the coefficient of $x^{2i}$ in the power series expansion of $f(x)$.
\begin{lema}
\label{lema:obstr-formula-simplified}
    Suppose that, with respect to \eqref{eqn:norm-inv-specific-iso}, $[\xi,t]\in[Th(\cpdt),G/O]$ has integral coordinates $(s_1,\dots,s_{t_{n,2l}^{'}})$ if $l$ is odd and $(y,s_1,\dots,s_{t_{n,2l}^{'}})$ if $l$ is even, where the coordinate $y$ corresponds to $\xi_{S^{2l}}$. Then 
    \begin{align}
    \label{eqn:obstr-final-formula-even-l}
        \sigma_{n,2l}^{DIFF}([\xi,t])=
        \begin{cases}
            -ya_{n,l}^{0}+\sum_{i=1}^{t_{n,2l}^{'}}s_i a_{n,l}^{i} & \text{if $l$ is even} \\
            \sum_{i=1}^{t_{n,2l}^{'}}s_ib_{n,l}^i & \text{if $l$ is odd}
        \end{cases}
    \end{align}
\end{lema}
\noindent\textbf{Proof: }The lemma is proved using the formula \eqref{eqn:obstr-formula-Hirz-simplified}. We prove the case where $n$ is even and $l\equiv0\mod4$, the remaining cases follow analogously. 

In this case, $\sN_{\partial}^{DIFF}(\cpdt)_{\text{free}}$ is isomorphic to $\ker\pi_{2l}(J)\oplus\ker[\Sigma^{2l}\cp^n,J]$. As above, we denote the generators by $\xi_{S^{2l}},\xi_1,\dots,\xi_{t_{n,2l}^{'}}$, where
\begin{align*}
    \xi_{S^{2l}}&=\denom\left(\frac{B_{\frac{l}{2}}}{2l}\right)\eta_{2l}, \\
    \xi_i&=m_{1,i}g_{\R}^{\lfloor\frac{l}{4}\rfloor}\mu_2+\dots+m_{t_{n,2l}^{'},i}g_{\R}^{\lfloor\frac{l}{4}\rfloor}\mu_2\mu_0^{t_{n,2l}^{'}-1},\hspace{3pt}1\leq i\leq t_{n,2l}^{'}, 
\end{align*}
$\eta_{2l}$ being the generator of $\pi_{2l}(BO)\cong\Z$. The result for $\xi_{S^{2l}}$ follows from \cite[Theorem 3.7]{Adams} in conjunction with \cite[Theorem 1.1]{quillen1971adams} (which solves the ambiguity for $2l\equiv0\mod8$). 

This, together with Lemma~\ref{cherngener}, implies that 
\begin{align*}
    ph(\xi_i)&=m_{1,i}z^{\times2l}\times h(x)+\dots+m_{t_{n,2l}^{'},i}z^{\times2l}\times h(x)^{t_{n,2l}^{'}}
\end{align*}
The cup products over suspension spaces vanish and the odd-degree Chern classes of $\xi_i\otimes\C$ are zero. Furthermore, $p_j(\xi_i)=(-1)^jc_{2j}(\xi_i\otimes\C)$, so we get from the definition of the Chern character \cite[3.22]{karoubi} that
\[
ph(\xi_i)=\sum_{j\geq1}(-1)^{j+1}\frac{p_j(\xi_i)}{(2j-1)!},
\]
which implies that $p_j(\xi_i)=(-1)^{j+1}(2j-1)!ph(\xi_i)_{2j}$. Additionally, from \cite[Theorem 3.8]{levine} we get the following:
\[
    p_{\frac{l}{2}}(\xi_{S^{2l}})=(-1)^{\frac{l}{2}}\frac{3-(-1)^l}{2}(l-1)!\denom\left(\frac{B_{\frac{l}{2}}}{2l}\right),\hspace{5pt} p_j(\xi_{S^{2l}})=0\text{ for }j\neq\frac{l}{2}.
\]

Again, using the triviality of cup products, we get from \cite[19-C]{milnor1974characteristic} that
\begin{align*}
    \sL^{-1}(\xi_i)&=1-\sum_{j\geq1}2^{2j}(2^{2j-1}-1)\frac{B_j}{2j!}p_j(\xi_i) \\ 
    \sL^{-1}(\xi_{S^{2l}})&=1-2^{l}(2^{l-1}-1)\frac{B_{\frac{l}{2}}}{l!}p_{\frac{l}{2}}(\xi_{S^{2l}})
\end{align*}
and substituting for $p_{\frac{l}{2}}(\xi_{S^{2l}})$, $p_j(\xi_i)$ and $ph(\xi_i)$ we obtain
\begin{align*}
    \sL^{-1}(\xi_i)&=1+\sum_{j\geq1}(-1)^j2^{2j}(2^{2j-1}-1)\frac{B_j}{2j}\left(\sum_{k=1}^{t_{n,2l}^{'}} m_{k,i}(z^{\times2l}\times h(x)^k)_{2j}\right) \\
    \sL^{-1}(\xi_{S^{2l}})&=1-(-1)^{\frac{l}{2}}2^{l+1}(2^{l-1}-1)\frac{3-(-1)^l}{2}\num\left(\frac{B_{\frac{l}{2}}}{2l}\right)
\end{align*}

In \eqref{eqn:obstr-formula-Hirz-simplified}, the maps $\pi$ and $\proj_1$ are projections on the first and second summand, respectively, in the splitting
\[
    \cp^n\times S^{2l}\simeq Th(\cpdt)\vee\cp^n;
\]
see also Remark~\ref{rem:splitting-cp-times-S-thom-space}. In cohomology, we have $H^*(\cp^n\times S^{2l})\cong H^*(\cp^n)\otimes H^*(S^{2l})$, and the direct summands $\widetilde{H}^*(Th(\cpdt))$ and $\widetilde{H}^*(\cp^n)$ correspond to the subrings $\widetilde{H}^*(\cp^n)\times y$ and $\widetilde{H}^*(\cp^n)\times 1$, respectively. Finally,
\[
    \sL(\cp^n)=\left(\frac{x}{tanh(x)}\right)^{n+1}
\]
where $x$ generates $H^*(\cp^n)$ (see, e.g., the proof of \cite[19.4]{milnor1974characteristic}). Substituting these results into \eqref{eqn:obstr-formula-Hirz-simplified} we obtain the formula \eqref{eqn:obstr-final-formula-even-l}. \qed
\begin{com}
    The computation of Pontryagin classes and $\sL$-classes in formula \eqref{eqn:obstr-formula-Hirz-simplified} is another place where the calculations are simplified compared to the case when $k=0$ studied in \cite{Brumfiel-(1970))}. Since cup products do not vanish in $H^*(\cp^n;\Q)$, some non-linear terms appear, and $\sigma_{n,0}^{DIFF}$ is is in general only polynomial in the coordinates of $[\xi,t]$. In particular, $\sigma_{n,0}^{DIFF}$ is not a homomorphism.
\end{com}
\noindent\textbf{Proof of Theorem~\ref{mainthm1}: }Choosing some splitting in Lemma~\ref{suspensg/o} and generators of $\ker[\Sigma^{2l}\cp^n,J]$, see \eqref{eqn:any-choice-of-generators-ker-J}, we obtain a specific isomorphism
\[
    \coker[\Sigma^{k}\cp^n,\Omega J]\oplus\Z^{t_{n,k}^{'}}\xrightarrow{\cong}[\Sigma^k\cp^n,G/O]
\]
This gives us, by Corollary~\ref{cor:norm-inv-splitting}, the following isomorphism:
\begin{align}
\label{eqn:norm-inv-iso-generators-1}
    T_{n,k}^{'}\oplus\Z^{t_{n,k}^{'}+\epsilon_{k}} \xrightarrow{\cong}[Th(\cpd),G/O]\cong\sN_{\partial}^{DIFF}(\cpd).
\end{align}

For $2n+k\equiv2\mod4$, $\sigma_{n,k}^{DIFF}$ is a homomorphism into $\Z_2$, so 
\[
    \ker\sigma_{n,k}^{DIFF}\cong\Z^{t_{n,k}^{'}+\epsilon_k} \oplus T_{n,k} 
\]
for some torsion subgroup $T_{n,k}$. Note that the free part of $\ker\sigma_{n,k}^{DIFF}$ is, in general, a subgroup of index $1$ or $2$ in the free part of \eqref{eqn:norm-inv-iso-generators-1}. The statement of the theorem follows from \eqref{SESeven} by observing the relation between $t_{n,k}$ and $t_{n,k}^{'}$; see Remark~\ref{com:relation-exponents}. This argument was already hinted at in Remark~\ref{homotgroupG/O}.

For $2n+k\equiv0\mod4$, composing \eqref{eqn:norm-inv-iso-generators-1} with the surgery obstruction map we obtain:
\[
    T_{n,k}^{'}\oplus\Z^{t_{n,k}^{'}+\epsilon_k} \xrightarrow{\cong} \sN_{\partial}^{DIFF}(\cpd)\xrightarrow{\sigma_{n,k}^{DIFF}}\Z.
\]
From Lemma~\ref{lem:surgery-obstruction-signature-formula} we know that this composition is zero on $T_{n,k}^{'}$ (for any choice of splitting in Lemma~\ref{suspensg/o}). From Lemma~\ref{lema:obstr-formula-simplified} it is obvious that this composition is given on $\Z^{t_{n,k}^{'}+\epsilon_k}$ by a linear polynomial with non-zero image. Thus, 
\[
    \ker\sigma_{n,k}^{DIFF}\cong T_{n,k}^{'}\oplus\Z^{t_{n,k}^{'}+\epsilon_k-1},
\]
where the isomorphism is given by choosing specific generators of the subgroup of roots of \eqref{eqn:obstr-final-formula-even-l}. Again, the statement of the theorem follows from \eqref{SESeven} by observing the relation between $t_{n,k}$ and $t_{n,k}^{'}$ from Remark~\ref{com:relation-exponents}.

Finally, if $2n+k$ is odd, we obtain from \eqref{SESodd} the following short exact sequence:
\[
    0\to\coker(\sigma_{n,k+1}^{DIFF})\to\sS_{\partial}^{DIFF}(\cpd)\to\sN_{\partial}^{DIFF}(\cpd)\to0.
\]
Here, the group $\sN_{\partial}^{DIFF}(\cpd)$ is finite by Lemma~\ref{suspensg/o}. Since $\coker(\sigma_{n,k+1}^{DIFF})$ is either a quotient of $\Z_2$, or a quotient of $\Z$ by some non-zero subgroup, it is finite as well. This completes the proof of the theorem. \qed

%% file: splittinginv.tex
\section{The Forgetful Map $F_{\cpd}$}
\label{nonsmooth}
In this section, we work towards a proof of Theorems \ref{mainthm2}, \ref{mainthm3} and Corollary~\ref{maincoroll}. There is a surgery exact sequence analogous to the smooth one for $TOP$-manifolds \cite[Section 11.6]{Lueck-Macko(2024)} and for $\cpdt$ it splits analogously to (\ref{SESeven}). We obtain a commutative diagram 
\begin{equation}
\label{smoothingcommut}
    \begin{tikzcd}
        0 \arrow{r}{} & \sS_{\partial}^{TOP}(\cpdt) \arrow{r}{\eta_{n,2l}^{TOP}} & \sN_{\partial}^{TOP}(\cpdt) \arrow{r}{\sigma_{n,2l}^{TOP}} & L_{2(n+l)}(\Z) \\
        0 \arrow{r}{} & \sS_{\partial}^{DIFF}(\cpdt) \arrow{r}{\eta_{n,2l}^{DIFF}} \arrow{u}{F_{\cpdt}} & \sN_{\partial}^{DIFF}(\cpdt) \arrow{r}{\sigma_{n,2l}^{DIFF}} \arrow{u}{} & L_{2(n+l)}(\Z) \arrow[equal]{u} 
    \end{tikzcd}
\end{equation}
There is an isomorphism
\begin{align*}
   \textstyle\sN_{\partial}^{TOP}(\cpdt)&\xrightarrow{\cong}\displaystyle\bigoplus_{i=0}^{t_{n,2l}^{'}}\textstyle L_{4i}(\Z)\oplus\displaystyle\bigoplus_{i=0}^{n+1-t_{n,2l}^{'}}\textstyle L_{4i+2}(\Z)\cong\textstyle \Z^{t_{n,2l}^{'}}\oplus\Z_2^{n+1-t_{n,2l}^{'}}
\end{align*}
given by splitting invariants $\overline{\sigma}_{m,2l}$, $0\leq m\leq n$. These are the surgery obstructions $\sigma_{m,2l}^{TOP}$ along the standardly embedded $\cp^m\times \text{D}^{2l}$ $(m\leq n)$. More precisely, $\overline{\sigma}_{m,2l}$ is given by the composition $$\textstyle\sN_{\partial}^{TOP}(\cpdt)\xrightarrow{\sN_{\partial}^{TOP}(i_{m,n})}\sN_{\partial}^{TOP}(\cp^m\times \text{D}^{2l})\xrightarrow{\sigma_{m,2l}^{TOP}}L_{2(m+l)}(\Z).$$ The map $\eta_{n,2l}^{TOP}$ corresponds to the inclusion of the subgroup of elements with the last splitting invariant $\overline{\sigma}_{n,2l}=0$; see \cite[14.C]{Wall(1999)}.

\noindent\textbf{Proof of Theorems~\ref{mainthm2}, \ref{mainthm3}: }Fix a particular splitting in Lemma~\ref{suspensg/o}. Choose the generators of $\ker[\Sigma^{2l}\cp^n,J]$ as described in Remark~\ref{kerJgenerstandembed}. Finally, choose some generators of $\ker\sigma_{n,2l}^{DIFF}$. These choices give us compatible isomorphisms:
\[
\begin{tikzcd}
    \Z^{t_{n,2l}}\oplus T_{n,2l} \arrow[d,phantom,sloped,"\cong"] \arrow[r,hookrightarrow,"E_{n,l}"] & \Z^{t_{n,2l}^{'}+\epsilon_{2l}}\oplus T_{n,2l}^{'} \arrow[d,phantom,sloped,"\cong"] \\
    \sS_{\partial}^{DIFF}(\cpdt) \arrow{r}{\eta_{n,2l}^{DIFF}} & \sN_{\partial}^{DIFF}(\cpdt)
\end{tikzcd}
\]
see also Corollary~\ref{cor:norm-inv-splitting}. Here, $E_{n,l}$ is given by a $(2\times2)$-block matrix
\[
\begin{pmatrix}
    P_{n,l} & 0 \\
    Q_{n,l} & R_{n,l}
\end{pmatrix}
\]
where $P_{n,l}:\Z^{t_{n,2l}}\to\Z^{t_{n,2l}^{'}+\epsilon_{2l}}$. The map $\eta_{n,l}^{DIFF}$ corresponds to the inclusion of the subgroup $\ker\sigma_{n,l}^{DIFF}$ in $\sN_{\partial}^{DIFF}(\cpdt)$; see \eqref{SESeven}. Thus, the matrix $E_{n,l}$, and specifically $P_{n,l}$, is determined by a choice of generators of $\ker\sigma_{n,l}^{DIFF}$. If $n+l$ is odd, the free part of $\ker\sigma_{n,l}^{DIFF}$ is either the entire free part of $\sN_{\partial}^{DIFF}(\cpdt)$, or a subgroup of index $2$, as already mentioned in the proof of Theorem~\ref{mainthm1}. Thus, either the matrix $P_{n,l}$ is the identity matrix, or $\coker P_{n,l}\cong\Z_2$.

From Diagram~\eqref{smoothingcommut} we see that the matrix $A_{n,l}$ from \eqref{eqn:matrix-forgetful-map} can be obtained by computing the integral splitting invariants of the composition
\begin{align}
\label{eqn:splitting-inv-on-image-P}
    \Z^{t_{n,2l}}\xrightarrow{P_{n,l}} \Z^{t_{n,2l}^{'}+\epsilon_{2l}}\subseteq\sN_{\partial}^{DIFF}(\cpdt)\to\sN_{\partial}^{TOP}(\cpdt).
\end{align}
The maps on normal invariants from \eqref{smoothingcommut} commute with maps induced by $i_{m,n}$. In particular, we obtain a diagram
\begin{equation*}
    \begin{tikzcd}
        \sN_{\partial}^{TOP}(\cpdt) \arrow{r}{\sN_{\partial}^{TOP}(i_{m,n})} &[5em] \sN_{\partial}^{TOP}(\cp^m\times \text{D}^{2l}) \arrow{r}{\sigma_{m,2l}^{TOP}} & L_{2(m+l)}(\Z) \\
        \sN_{\partial}^{DIFF}(\cpdt) \arrow{r}{\sN_{\partial}^{DIFF}(i_{m,n})} \arrow{u}{} & \sN_{\partial}^{DIFF}(\cp^m\times \text{D}^{2l}) \arrow{r}{\sigma_{m,2l}^{DIFF}} \arrow{u}{} & L_{2(m+l)}(\Z) \arrow[equal]{u} 
    \end{tikzcd}
\end{equation*}
so the integral splitting invariants $\overline{\sigma}_{m,2l}$ of the composition \eqref{eqn:splitting-inv-on-image-P} can be computed in terms of 
\[
\sN_{\partial}^{DIFF}(\cpdt)\xrightarrow{\sN_{\partial}^{DIFF}(i_{m,n})}\sN_{\partial}^{DIFF}(\cp^m\times \text{D}^{2l}) 
\]
and $\sigma_{m,2l}^{DIFF}$ for $0\leq m<n$ and $m+l$ even (so that $\overline{\sigma}_{m,2l}\in\Z$). 

Since we have chosen the generators of $\ker[\Sigma^{2l}\cp^n,J]$ as in Remark~\ref{kerJgenerstandembed}, the map $\sN_{\partial}^{DIFF}(i_{m,n})$ is given on the free parts by the $((t_{m,2l}^{'}+\epsilon_{2l})\times(t_{n,2l}^{'}+\epsilon_{2l}))$-matrix 
\[
    \begin{pNiceMatrix}[nullify-dots, columns-width=6pt]
        1      & 0      & & \Cdots[shorten=15pt] & 0 
        & 0      &       & & \Cdots[shorten=15pt] & 0 \\
        0      & \Ddots[shorten=20pt] & & & \Vdots[shorten=10pt] 
        &       & \Ddots[shorten=18pt] & & & \Vdots[shorten=10pt] \\
        \Vdots &        & &        &  
        & \Vdots &        & &        &  \\
           &        & &        & 0 
        &        &        & &        &  \\
        0      & \Cdots[shorten=5pt] & & 0 & 1 
        & 0      & \Cdots[shorten=5pt] & &  & 0
    \end{pNiceMatrix}.
\]
This representation is with respect to the basis $(\xi_1,\dots,\xi_{t_{n,l}^{'}})$ with, in addition, $\xi_{S^{2l}}$ as the first basis element when $l$ is even. 

The surgery obstructions $\sigma_{m,2l}^{DIFF}$ for $0\leq m<n$ and even $m+l$ are given by \eqref{eqn:obstr-final-formula-even-l}. It follows that the integral splitting invariants $\overline{\sigma}_{m,2l}$, $0\leq m<n$ of the image of $P_{n,l}$ are given by a lower triangular matrix $A_{n,l}^{'}$, with $(j,i)$-th entry (for $0\leq j,i\leq t_{n,l}-1$, $j\geq i$) given by $a_{2j,l}^i$ from \eqref{eqn:coeff-a_n,l} and \eqref{eqn:coeff-a_n,l-sphere} if $l$ is even, resp. by $b_{2j+1,l}^{i+1}$ from \eqref{eqn:coeff-b_n,l} if $l$ is odd. If $n+l$ is even, it has an additional $(t_{n,l}+1)$st column of zeros. The matrix $A_{n,l}$ from \eqref{eqn:matrix-forgetful-map} is thus given by $A_{n,l}=A_{n,l}^{'}\cdot P_{n,l}$. \qed
\newline
\newline
\noindent\textbf{Proof of Corollary~\ref{maincoroll}:} To obtain the necessary condition $(1)$ of Corollary \ref{maincoroll}, we determine the image of $A_{n,l}$ by congruences. The image of $A_{n,l}^{'}$ is characterized by expressing the integral splitting invariants $\overline{\sigma}_{m,2l}$ in terms of lower ones, as is done for the integral invariants of $\cp^n$ in \cite[4.5]{Little1989}. Since $A_{n,l}=A_{n,l}^{'}\cdot P_{n,l}$, we have $\im A_{n,l}\subseteq\im A_{n,l}^{'}$, and we obtain a necessary condition. When the total dimension $2(n+l)$ is $0\mod4$, this can be improved by imposing the condition $\overline{\sigma}_{n,2l}^{DIFF}=0$ which describes $\im P_{n,l}$, so we have a complete characterization of $\im A_{n,l}$.

For example, if $m_1\xi_1+m_2\xi_2+m_3\xi_3\in\sN_{\partial}^{DIFF}(\cp^6\times \text{D}^2)$ we get from Table~\ref{signobstrtable1}: 
\[
    \overline{\sigma}_{5,2}=-662m_1-632m_2-992m_3, \overline{\sigma}_{3,2}=44m_1+28m_2 \text{ and } \overline{\sigma}_{1,2}=-2m_1.
\]
This gives: 
\[
    \overline{\sigma}_{1,2}=-2m_1, \overline{\sigma}_{3,2}=-22\overline{\sigma}_{1,2}+ 28m_2\text{ and }\overline{\sigma}_{5,2}=-\frac{158}{7}\overline{\sigma}_{3,2}-\frac{1159}{7}\overline{\sigma}_{1,2}-992m_3.
\]
Since $m_1,m_2,m_3\in\Z$, we obtain the following congruences: 
\begin{align*}
    \overline{\sigma}_{1,2}&\equiv0\mod2, \\
    \overline{\sigma}_{3,2}+8\overline{\sigma}_{1,2} &\equiv0\mod28 \\
    \overline{\sigma}_{5,2}+\frac{158\overline{\sigma}_{3,2} +1159\overline{\sigma}_{1,2}}{7}&\equiv0\mod992
\end{align*} 

To obtain a sufficient condition, we need to take into account the $\Z_{2}$-valued splitting invariants. Suppose that $[f,\partial f]\in\sS_{\partial}^{TOP}(\cpdt)$ is a multiple of $2$ and its integral splitting invariants satisfy the required congruences. Since $\overline{\sigma}_{m,2l}$ are homomorphisms and $[f,\partial f]$ is a multiple of $2$, its $\Z_2$-valued splitting invariants are zero. Thus these invariants are definitely in $\im B_{n,l}$ and $\im C_{n,l}$. 

As discussed above, the congruences characterize $\im A_{n,l}$ when $n+l$ is even and $\im A_{n,l}^{'}$ otherwise (in these cases the congruences do not take into account $\im P_{n,l}$). The corollary is thus proved for even $n+l$ ($[f,\partial f]$ belongs to $\im F_{n,l}$). For odd $n+l$, since we know that the integral invariants satisfy the required congruences and the $\Z_2$-invariants are zero, we have the following:
\[
    \eta_{n,2l}^{TOP}([f,\partial f])\in\im(\sN_{\partial}^{DIFF}(\cpdt)\to\sN_{\partial}^{TOP}(\cpdt)).
\]
Let $[\xi_f,t_f]\in\sN_{\partial}^{DIFF}(\cpdt)$ denote an element in the preimage of $\eta_{n,2l}^{TOP}([f,\partial f])$. We know that $\sigma_{n,2l}^{TOP}(\eta_{n,2l}^{TOP}([f,\partial f]))=0$. This implies that $\sigma_{n,2l}^{DIFF}([\xi_f,t_f])=0$ by \eqref{smoothingcommut}, so $[\xi_f,t_f]\in\im\eta_{n,2l}^{DIFF}$.
Thus $[(f,\partial f)]\in\im F_{\cpdt}$. 
\qed

%% file: appendix.tex
\section{}
Let $H$ be the class of the canonical complex line bundle over $\cp^n$ in $K^0(\cp^n)$. Denote:
    \begin{align*}
        \mu&=H-1\in\widetilde{K}^0(\textstyle{\cp^n}) \\       \mu_0&=r(\mu)\in\widetilde{KO}^0(\textstyle{\cp^n}) \\
        \mu_i&=r(g^i\cdot\mu)\in\widetilde{KO}^{-2i}(\textstyle{\cp^n}) \text{ for }i=1,2,3
    \end{align*}
    where $r:\widetilde{K}^*(X)\to\widetilde{KO}^*(X)$ is the real restriction map and $g\cdot\mu$ denotes the Bott periodicity isomorphism $\widetilde{K}^0(\cp^n)\cong\widetilde{K}^{-2}(\cp^n)$ where $g=H-1\in K^0(\cp^1)\cong K^0(S^2)$. According to Fujii \cite{Fujii}, the values of the reduced cohomology ring $\widetilde{KO}^*(\cp^n)$ are as follows:
\begin{lema}\cite[Thm 2]{Fujii}
\label{KOCP}
    \begin{itemize}
        \item[0) ]$KO^0(\cp^n)$ is a truncated polynomial ring generated by $\mu_0$ subject to the relations 
        \begin{align*}
            \mu_0^{t+1}=0 & \text{ if }n=2t \\
            2\mu_0^{2t+1}=0 \text{ and } \mu_0^{2t+2}=0 & \text{ if }n=4t+1 \\
            \mu_0^{2t+2}=0 & \text{ if }n=4t+3
        \end{align*}
        \item[1) ]$\widetilde{KO}^{-1}(\cp^n)=0$
        \item[2) ]$\widetilde{KO}^{-2}(\cp^n)$ is the free module generated by $\mu_1,\mu_1\mu_0,\dots,\mu_1\mu_0^{t-1}$ where $t=\left[\frac{n}{2}\right]$, and also $\mu_1\mu_0^t$ if $n\equiv1\mod4$ or $\sigma$ with the relation $2\sigma=\mu_1\mu_0^t$ if $n\equiv3\mod4$
        \item[3) ]$\widetilde{KO}^{-3}(\cp^n)=0$ if $n\neq4t+3$ and $\Z_2$ for $n=4t+3$
        \item[4) ]$\widetilde{KO}^{-4}(\cp^n)$ is the free module generated by $\mu_2,\mu_2\mu_0,\dots,\mu_2\mu_0^{t-1}$ where $t=\left[\frac{n}{2}\right]$, and also $\mu_2\mu_0^t$ with the relation $2\mu_2\mu_0^t=0$ if $n\equiv3\mod4$ 
        \item[5) ]$\widetilde{KO}^{-5}(\cp^n)=0$
        \item[6) ]$\widetilde{KO}^{-6}(\cp^n)$ is the free module generated by $\mu_3,\mu_3\mu_0,\dots,\mu_3\mu_0^{t-1}$ where $t=\left[\frac{n}{2}\right]$, and also $\mu_3\mu_0^t$ if $n\equiv3\mod4$ or $\tau$ with the relation $2\tau=\mu_3\mu_0^t$ if $n\equiv1\mod4$
        \item[7) ]$\widetilde{KO}^{-7}(\cp^n)=0$ if $n\neq4t+1$ and $\Z_2$ for $n=4t+1$
    \end{itemize}
\end{lema}
This theorem determines the entire cohomology ring $\widetilde{KO}^0(\cp^n)$, since we have the real Bott periodicity isomorphisms $\widetilde{KO}^{-i}(\cp^n)\xrightarrow{\cong}\widetilde{KO}^{-i-8}(\cp^n)$ given by a product with the generator $g_{\R}\in\widetilde{KO}^{-8}(S^0)$.

\newpage

\begin{table}[ht!]
        \centering
        \begin{tabular}{|c|c|c|c|c|c|c|}
            \hline
            \diaghead{\theadfont aaaaaaaaaa}%
            {$l$}{$n$} & $2$ & $3$ & $4$ \\
            \hline
            $1$ & $
            \begin{pmatrix}
                -2
            \end{pmatrix}\cdot P_{1,2}
            $ &
            $
            \begin{pmatrix}
                -2 & 0
            \end{pmatrix}\cdot P_{1,3}\hspace{2pt}{}^{(i)}
            $
            &
            $
            \begin{pmatrix}
                -2 & 0 \\
                44 & 28
            \end{pmatrix}\cdot P_{1,4}
            $
            \\
            \hline
            $2$ & 
            $
            \begin{pmatrix}
                2 & 0
            \end{pmatrix}\cdot P_{2,2}\hspace{2pt}{}^{(ii)}
            $ 
            &
            $
            \begin{pmatrix}
                2 & 0 \\
                2 & 28 
            \end{pmatrix}\cdot P_{2,3}$
            & 
            $
            \begin{pmatrix}
                2 & 0 & 0 \\
                2 & 28 & 0 
            \end{pmatrix}\cdot P_{2,4}\hspace{2pt}{}^{(iii)}
            $
            \\
            \hline
            $3$ & 
            $
            \begin{pmatrix}
                28
            \end{pmatrix}\cdot P_{3,2}
            $
            &
            $
            \begin{pmatrix}
                28 & 0
            \end{pmatrix}\cdot P_{3,3}\hspace{2pt}{}^{(iv)}
            $
            & 
            $
            \begin{pmatrix}
                28 & 0 \\
                -752 & -992
            \end{pmatrix}\cdot P_{3,4}$
            \\
            \hline
            \diaghead{\theadfont aaaaaaaaaa}%
            {$l$}{$n$} & \multicolumn{2}{c|}{$5$} & $6$ \\
            \hline
            $1$ & 
            \multicolumn{2}{c|}{$
            \begin{pmatrix}
                -2 & 0 & 0 \\
                44 & 28 & 0
            \end{pmatrix}
            \cdot P_{1,5}\hspace{2pt}{}^{(v)}$}
            &
            $
            \begin{pmatrix}
                -2 & 0 & 0 \\
                44 & 28 & 0 \\
                -662 & -632 & -992
            \end{pmatrix}\cdot P_{1,6}$
            \\
            \hline
            $2$ &
            \multicolumn{2}{c|}{$
            \begin{pmatrix}
                2 & 0 & 0 \\
                2 & 28 & 0 \\
                2 & -408 & -496 
            \end{pmatrix}\cdot P_{2,5}$}
            & 
            $
            \begin{pmatrix}
                2 & 0 & 0 & 0 \\
                2 & 28 & 0 & 0 \\
                2 & -408 & -496 & 0 
            \end{pmatrix}\cdot P_{2,6}\hspace{2pt}{}^{(vi)}
            $
            \\
            \hline
            $3$ &
            \multicolumn{2}{c|}{$
            \begin{pmatrix}
                28 & 0 & 0 \\
                -752 & -992 & 0
            \end{pmatrix}
            \cdot P_{3,5}\hspace{2pt}{}^{(vii)}
            $}
            & 
            $
            \begin{pmatrix}
                28 & 0 & 0 \\
                -752 & -992 & 0 \\
                10760 & 10208 & 8128
            \end{pmatrix}\cdot P_{3,6}$
            \\
            \hline
        \end{tabular}
        \caption{The matrices $A_{n,l}$ for $1\leq l\leq3$, $2\leq n\leq6$} 
        \label{tabobstrtors}
    \end{table}
${}^{(i)}$ The columns of $P_{1,3}$ generate the solutions of $44x+28y=0$ over the integers.

${}^{(ii)}$ The columns of $P_{2,2}$ generate the solutions of $28x+2a=0$ over the integers.

${}^{(iii)}$ The columns of $P_{2,4}$ generate the solutions of $-408x-496y+2a=0$ over the integers.

${}^{(iv)}$ The columns of $P_{3,3}$ generate the solutions of $-752x-992y=0$ over the integers.

${}^{(v)}$ The columns of $P_{1,5}$ generate the solutions of $-662x-632y-992z=0$ over the integers.

${}^{(vi)}$ The columns of $P_{2,6}$ generate the solutions of $5796x+5616y+8128z+2a=0$ over the integers, e.g.: \\
        \[P_{2,6}=
        \begin{pmatrix}
            1 & 0 & 0 \\
            0 & 1 & 0 \\
            0 & 1 & 1 \\
            -2898 & -6872 & -4064
        \end{pmatrix}
        ,\quad
        A=
        \begin{pmatrix}
            -5796 & -13744 & -8128 \\
            -5768 & -13744 & -8128 \\
            -6204 & -14240 & -8128
        \end{pmatrix}
        \]

${}^{(vii)}$ The columns of $P_{3,5}$ generate the solutions of $10760x+10208y+8128z=0$ over the integers.

The remaining matrices correspond to cases where $n+l$ is odd and are either identity matrices, or matrices with cokernel isomorphic to $\Z_2$.

\newpage

\begin{table}[ht!]
        \centering
        \begin{tabular}{|c|c|}\hline
             Manifold & Congruences \\ \hline
             \thead{$X^{14}\simeq_{\partial}\cp^6\times D^2$ \\ } & \thead{$\overline{\sigma}_{1,2}\equiv0\mod2, \overline{\sigma}_{3,2}+22\overline{\sigma}_{1,2}\equiv0\mod28$ \\
             $\overline{\sigma}_{5,2}+\dfrac{158\overline{\sigma}_{3,2}+1159\overline{\sigma}_{1,2}}{7}\equiv0\mod992$} \\ 
             \hline
             \thead{$X^{16}\simeq_{\partial}\cp^6\times D^4$} & \thead{$\overline{\sigma}_{0,4}\equiv0\mod2, \overline{\sigma}_{2,4}-\overline{\sigma}_{0,4}\equiv0\mod28$ \\
             $\overline{\sigma}_{4,4}+\dfrac{102\overline{\sigma}_{2,4}-109\overline{\sigma}_{0,4}}{7} \equiv0\mod496$ \\
             $\dfrac{351\overline{\sigma}_{4,4}}{31} + \dfrac{6443\overline{\sigma}_{0,4} - 9117\overline{\sigma}_{2,4}}{217}\equiv0\mod 8128$} \\ 
             \hline
             \thead{$X^{18}\simeq_{\partial}\cp^6\times D^6$} & \thead{$\overline{\sigma}_{1,6}\equiv0\mod28, \overline{\sigma}_{3,6}+\dfrac{188}{7}\overline{\sigma}_{1,6}\equiv0\mod992$ \\
             $\overline{\sigma}_{5,6}-\dfrac{23418}{217}\overline{\sigma}_{1,6}+\dfrac{319}{31}\overline{\sigma}_{3,6}\equiv0\mod8128$} \\ \hline
        \end{tabular}
        \caption{Congruences of Splitting Invariants}
        \label{tabsplitinv}
    \end{table}

\newpage

\begin{table}[ht!]
    \centering
    \begin{tabular}{|c|}
        \hline
         $\cp^{18}\colon$ \\ \hline \thead{$\xi_1=24\mu_0+98\mu_0^2+111\mu_0^3+54\mu_0^4+12\mu_0^5+\mu_0^6,$ \\ $\xi_2=240\mu_0^2+380\mu_0^3+28\mu_0^4+207\mu_0^5+18443\mu_0^6+5\mu_0^7+11452\mu_0^8+4110\mu_0^9,$ \\ $\xi_3=504\mu_0^3+234\mu_0^4+91\mu_0^5+64628\mu_0^6+\mu_0^7+1774\mu_0^8+12497\mu_0^9,$ \\ $\xi_4=480\mu_0^4+16\mu_0^5+14346\mu_0^6+\mu_0^7+2742\mu_0^8+5350\mu_0^9,$ \\ $\xi_5=264\mu_0^5+5350\mu_0^6+17\mu_0^7+4950\mu_0^8+20911\mu_0^9,$ \\ $\xi_6=65520\mu_0^6+12\mu_0^7+2072\mu_0^8+4145\mu_0^9,$\\ 
         $\xi_7=24\mu_0^7+1346\mu_0^8+16535\mu_0^9, \xi_8=16320\mu_0^8+8272\mu_0^9, \xi_9=28728\mu_0^9$} \\ \hline
         
         $\Sigma^2\cp^{17}\colon$ \\ \hline \thead{$\xi_1=24\mu_1+196\mu_1\mu_0+333\mu_1\mu_0^2+216\mu_1\mu_0^3+ 60\mu_1\mu_0^4+ 6\mu_1\mu_0^5,$ \\ $\xi_2=240\mu_1\mu_0+ 318\mu_1\mu_0^2+ 220\mu_1\mu_0^3+ 59\mu_1\mu_0^4+ 21625\mu_1\mu_0^5+ 13\mu_1\mu_0^6+ 5480\mu_1\mu_0^7+ 11307\mu_1\mu_0^8,$ \\ $\xi_3=504\mu_1\mu_0^2+ 152\mu_1\mu_0^3+ 79\mu_1\mu_0^4+ 15646\mu_1\mu_0^5+ 13\mu_1\mu_0^6+ 10076\mu_1\mu_0^7+ 8295\mu_1\mu_0^8,$ \\ $\xi_4=480\mu_1\mu_0^3+ 218\mu_1\mu_0^4+ 13230\mu_1\mu_0^5+ 14\mu_1\mu_0^6+ 11780\mu_1\mu_0^7+ 8436\mu_1\mu_0^8,$ \\ $\xi_5=264\mu_1\mu_0^4+ 32628\mu_1\mu_0^5+ 23\mu_1\mu_0^6+ 1776\mu_1\mu_0^7+ 10038\mu_1\mu_0^8,$ \\ $\xi_6=65520\mu_1\mu_0^5+ 10\mu_1\mu_0^6+ 6392\mu_1\mu_0^7+ 2838\mu_1\mu_0^8,$ \\ $\xi_7=24\mu_1\mu_0^6+ 10864\mu_1\mu_0^7+ 10161\mu_1\mu_0^8, \xi_8=16320\mu_1\mu_0^7+ 9306\mu_1\mu_0^8, \xi_9=28728\mu_1\mu_0^8$} \\ \hline
         
         $\Sigma^4\cp^{18}\colon$ \\ \hline \thead{$\xi_1=240\mu_2+442\mu_2\mu_0+327\mu_2\mu_0^2+114\mu_2\mu_0^3+23154\mu_2\mu_0^4+17\mu_2\mu_0^5+14020\mu_2\mu_0^6+5354\mu_2\mu_0^7+7838\mu_2\mu_0^8,$ \\ $\xi_2=504\mu_2\mu_0+316\mu_2\mu_0^2+116\mu_2\mu_0^3+6181\mu_2\mu_0^4+23\mu_2\mu_0^5+6040\mu_2\mu_0^6+14724\mu_2\mu_0^7+2646\mu_2\mu_0^8,$ \\ $\xi_3=480\mu_2\mu_0^2+78\mu_2\mu_0^3+20981\mu_2\mu_0^4+2\mu_2\mu_0^5+11908\mu_2\mu_0^6+27046\mu_2\mu_0^7+197\mu_2\mu_0^8,$ \\ 
         $\xi_4=264\mu_2\mu_0^3+43592\mu_2\mu_0^4+14\mu_2\mu_0^5+7131\mu_2\mu_0^6+28212\mu_2\mu_0^7+4855\mu_2\mu_0^8,$ \\ 
         $\xi_5=65520\mu_2\mu_0^4+14\mu_2\mu_0^5+9493\mu_2\mu_0^6+23620\mu_2\mu_0^7+11572\mu_2\mu_0^8,$ \\ 
         $\xi_6=24\mu_2\mu_0^5+8148\mu_2\mu_0^6+22416\mu_2\mu_0^7+1057\mu_2\mu_0^8,$ \\
         $\xi_7=16320\mu_2\mu_0^6+7238\mu_2\mu_0^7+49\mu_2\mu_0^8, \xi_8=28728\mu_2\mu_0^7+2848\mu_2\mu_0^8, \xi_9=13200\mu_2\mu_0^8$} \\ \hline
         
         $\Sigma^6\cp^{17}\colon$ \\ \hline \thead{$\xi_1=240\mu_3+380\mu_3\mu_0+267\mu_3\mu_0^2+84\mu_3\mu_0^3+37370\mu_3\mu_0^4+17\mu_3\mu_0^5+8530\mu_3\mu_0^6+5484\mu_3\mu_0^7+11724\tau,$ \\ $\xi_2=504\mu_3\mu_0+234\mu_3\mu_0^2+92\mu_3\mu_0^3+18025\mu_3\mu_0^4+17\mu_3\mu_0^5+11959\mu_3\mu_0^6+23196\mu_3\mu_0^7+12630\tau,$ \\ $\xi_3=480\mu_3\mu_0^2+16\mu_3\mu_0^3+55025\mu_3\mu_0^4+2\mu_3\mu_0^5+10709\mu_3\mu_0^6+25396\mu_3\mu_0^7+492\tau,$ \\ $\xi_4=264\mu_3\mu_0^3+5350\mu_3\mu_0^4+18\mu_3\mu_0^5+15290\mu_3\mu_0^6+27348\mu_3\mu_0^7+3276\tau,$ \\ $\xi_5=65520\mu_3\mu_0^4+12\mu_3\mu_0^5+13021\mu_3\mu_0^6+26896\mu_3\mu_0^7+816\tau,$ \\ $\xi_6=24\mu_3\mu_0^5+1346\mu_3\mu_0^6+25752\mu_3\mu_0^7+3108\tau,$ \\ $\xi_7=16320\mu_3\mu_0^6+8272\mu_3\mu_0^7+126\tau, \xi_8=28728\mu_3\mu_0^7+3108\tau, \xi_9=13200\tau$} \\ \hline
    \end{tabular}
    \caption{\centering Generators of $\ker[\Sigma^{2l}\cp^n,J]\subseteq[\Sigma^{2l}\cp^n,G/O]$ \\ for $0\leq l\leq3$, $n=17,18$}
    \label{kerJgenertable}
\end{table}

\newpage

\begin{table}[ht!]
    \centering
    \begin{tabular}{|c|c|c|c|}
        \hline
        $\Sigma^8\cp^4$ & \thead{$\xi_1=g_{\R}(504\mu_0+398\mu_0^2)$ \\ 
        $\xi_2=480g_{\R}\mu_0^2$} & $\Sigma^{20}\cp^4$ & \thead{$\xi_1=g_{\R}^2(65520\mu_2+22\mu_2\mu_0)$ \\ 
        $\xi_2=24g_{\R}^2\mu_2\mu_0$} \\
        \hline
        $\Sigma^{10}\cp^4$ & 
        \thead{$\xi_1=g_{\R}(504\mu_1+316\mu_1\mu_0)$ \\ 
        $\xi_2=480g_{\R}\mu_1\mu_0$} & $\Sigma^{22}\cp^4$ & 
        \thead{$\xi_1=g_{\R}^2(65520\mu_3+20\mu_3\mu_0)$ \\ 
        $\xi_2=24g_{\R}^2\mu_3\mu_0$}  \\
        \hline
        $\Sigma^{12}\cp^4$ & \thead{$\xi_1=g_{\R}(480\mu_2+202\mu_2\mu_0)$ \\ 
        $\xi_2=264g_{\R}\mu_2\mu_0$} & $\Sigma^{24}\cp^4$ & \thead{$\xi_1=g_{\R}^3(24\mu_0+9518\mu_0^2)$ \\ 
        $\xi_2=16320g_{\R}^3\mu_0^2$} \\
        \hline
        $\Sigma^{14}\cp^4$ & 
        \thead{$\xi_1=g_{\R}(480\mu_3+140\mu_3\mu_0)$ \\ 
        $\xi_2=264g_{\R}\mu_3\mu_0$} & $\Sigma^{26}\cp^4$ & \thead{$\xi_1=g_{\R}^3(24\mu_1+2716\mu_1\mu_0)$ \\ 
        $\xi_2=16320g_{\R}^3\mu_1\mu_0$} \\
        \hline
        $\Sigma^{16}\cp^4$ & \thead{$\xi_1=g_{\R}^2(264\mu_0+27278\mu_0^2)$ \\ 
        $\xi_2=65520g_{\R}^2\mu_0^2$} & $\Sigma^{28}\cp^4$ & \thead{$\xi_1=g_{\R}^3(16320\mu_2+1034\mu_2\mu_0)$\\ 
        $\xi_2=28728g_{\R}^3\mu_2\mu_0$} \\ 
        \hline
        $\Sigma^{18}\cp^4$ & 
        \thead{$\xi_1=g_{\R}^2(264\mu_1+54556\mu_1\mu_0)$ \\ 
        $\xi_2=65520g_{\R}^2\mu_1\mu_0$} & $\Sigma^{30}\cp^4$ & \thead{$\xi_1=g_{\R}^3(16320\mu_3+2068\mu_3\mu_0)$ \\ 
        $\xi_2=28728g_{\R}^3\mu_3\mu_0$} \\
        \hline
    \end{tabular}
    \caption{\centering Generators of $\ker[\Sigma^{2l}\cp^n,J]\subseteq[\Sigma^{2l}\cp^n,G/O]$ \\ for $4\leq l\leq15$, $n=4$}
    \label{kerJgenertable-l-big-n-4}
\end{table}

\newpage
Next, we provide tables of Pontryagin classes and surgery obstructions. In these tables, $x^i$ denotes the generator of $H^i(Th(\cpdt);\Z)$. It is not meant to suggest that $x^{2i}$ is the square of $x^i$ (or any other cup product relations).
\begin{table}[ht!]
    \centering
    \begin{tabular}{|c|c|c|}
        \hline 
        $n$ & Total Pontryagin class & Surgery obstruction \\ \hline
        $1$ & \thead{$1+48m_1x^2$} & \thead{$-2m_1$} \\ \hline
        $3$ & \thead{$P_1-(2400m_1+1440m_2)x^4$} & \thead{$44m_1 + 28m_2$} \\ \hline
        $5$ & \thead{$P_3-1440m_2x^4$ \\ $+(91728m_1 +90720m_2 + 120960m_3)x^6$} & \thead{$-662m_1 - 632m_2 - 992m_3$} \\ \hline
        $7$ & \thead{$P_5-(3345600m_1+3346560m_2$ \\ $+3225600m_3+2419200m_4)x^8$} & \thead{$9368m_1 + 9424m_2 + 8192m_3 + 8128m_4$} \\ \hline
        $9$ & \thead{$P_7-2419200m_4x^8+(120812208m_1$ \\
        $+ 120811680m_2+121080960m_3$ \\ $+303367680m_4+191600640m_5)x^{10}$} & \thead{$-130922m_1 - 130832m_2 - 133952m_3 $ \\ $ - 360064m_4 - 261632m_5$} \\ \hline
        $11$ & \thead{$P_9-(4352493600m_1+1730255567040m_2$ \\ $+1253360908800m_3+1067934067200m_4$ \\ $+2615348736000m_5+2615348736000m_6)x^{12}$} & \thead{$1824836m_1 + 957659336m_2 + 693540608m_3$ \\ $+ 589897856m_4 + 1447377920m_5 + 1448424448m_6$} \\ \hline
        $13$ & \thead{$P_{11}-2615348736000m_6x^{12}+(156718729488m_1$ \\ $+157376539887840m_2+113961751263360m_3$ \\ $+96660521717760m_4+237707801210880m_5$ \\ $+476118010368000m_6+298896998400m_7)x^{14}$} & \thead{$-25420670m_1 - 30859813576m_2 $ \\ $ - 22345976608m_3 - 18943544704m_4 $ \\
        $- 46607296000m_5 - 93367150592m_6 - 67100672m_7$} \\ \hline
        $15$ & \thead{$P_{13}-(5642133590400m_1$ \\
        $+23469801891160320m_2$ \\ 
        $+32971386263961600m_3$ \\
        $+36418511127398400m_4$ \\ 
        $+18444933771264000m_5$ \\
        $+44342366035968000m_6$ \\ 
        $+28454994247680000m_7$ \\ 
        $+21341245685760000m_8)x^{16}$} & \thead{$354076208m_1 + 1959482017568m_2$ \\ $ + 2872567698432m_3 + 3205593070848m_4$ \\
        $ + 1412456558592m_5 + 3502171938816m_6$ \\ $ + 2588712566784m_7 + 1941802827776m_8$} \\ \hline
        $17$ & \thead{$P_{15}-21341245685760000m_8x^{16}$ \\ $+(203119137971568m_1$ \\
        $+11403499401188647200m_2$ \\ $+11592924132729260160m_3$ \\
        $+12554193285185610240m_4$ \\ $+8747245743106590720m_5$ \\
        $+6746234526839808000m_6$ \\ $+13028190253785907200m_7$ \\
        $+15327282651512832000m_8$ \\ $+20436376868683776000m_9)x^{18}$} & \thead{$-4931681234m_1 - 408268795673248m_2 $ \\
        $- 409912538767488m_3 - 443417372493568m_4$ \\ $ - 313344128500736m_5 - 226263421313024m_6 $ \\
        $- 464901718245376m_7 - 541916044886016m_8 $ \\ $ - 753623571759104m_9$} \\ \hline
    \end{tabular}
    \caption{\centering Total Pontryagin Classes and Surgery Obstructions of \\ $\xi=m_1\xi_1+\dots+m_{t_{n,2}^{'}}\xi_{t_{n,2}^{'}}\in[Th(\cp^n\times D^{2}),G/O]_{free}$}
    \label{signobstrtable1}
\end{table}

\newpage

\begin{table}[ht!]
    \centering
    \begin{tabular}{|c|c|c|}
        \hline 
        $n$ & Total Pontryagin class & Surgery obstruction \\ \hline
        $2$ & \thead{$1-1440m_1x^4$} & \thead{$28m_1 + 2y$} \\ \hline
        $4$ & \thead{$P_2+(55440m_1+60480m_2)x^6$} & \thead{$-408m_1 - 496m_2 + 2y$} \\ \hline
        $6$ & \thead{$P_4-(2022720m_1$ \\ $+2016000m_2+2419200m_3)x^8$} & \thead{$5796m_1 + 5616m_2 + 8128m_3 + 2y$} \\ \hline
        $8$ & \thead{$P_6+(73042992m_1+73047744m_2$ \\ $+71850240m_3+95800320m_4)x^{10}$} & \thead{$-81008m_1 - 81312m_2 $ \\ $- 73728m_3 - 130816m_4 + 2y$} \\ \hline
        $10$ & \thead{$P_8-(926141055840m_1$ \\ $+248648400000m_2+839091052800m_3$ \\ $+1743565824000m_4+2615348736000m_5)x^{12}$} & \thead{$512584428m_1 + 137377056m_2 + 464390784m_3$ \\ $ + 965136640m_4 + 1448424448m_5 + 2y$} \\ \hline
        $12$ & \thead{$P_{10}+(60222864955344m_1+16222735983168m_2$ \\ $+54481549224960m_3+113277395911680m_4$ \\ $170084846131200m_5+149448499200m_6)x^{14}$} & \thead{$-11297842824m_1 - 3045962000m_2$ \\
        $- 10217829888m_3- 21246888192m_4$ \\ $- 31906572288m_5 - 33550336m_6 + 2y$} \\ \hline
        $14$ & \thead{$P_{12}-(20868702494428800m_1$ \\ $+8587799962560000m_2+17860060842624000m_3$ \\ $+14086351319040000m_4+19562808545280000m_5$ \\ $+10670622842880000m_6+21341245685760000m_7)x^{16}$} & \thead{$1837020065524m_1 + 764738647952m_2$ \\ 
        $+ 1569167127104m_3+ 1165483605760m_4$ \\ $+ 1605486011392m_5 + 970733662208m_6 $ \\
        $ + 1941802827776m_7 + 2y$} \\ \hline
        $16$ & \thead{$P_{14}+(4901857193340714096m_1$ \\ $+6514797972506185152m_2$ \\
        $+12170329845063586560m_3$ \\ $+11680218932456954880m_4$ \\
        $+10620124120276070400m_5$ \\ $+9664703227481702400m_6$ \\
        $+5960609920032768000m_7$ \\ $+10218188434341888000m_8)x^{18}$} & \thead{$-170191084996064m_1$ \\
        $- 235866025323072m_2$ \\
        $- 439760842491904m_3$ \\
        $- 423815802886656m_4$ \\ 
        $- 382076161564672m_5$ \\
        $- 350899867107328m_6 $ \\
        $- 208803325739008m_7$ \\
        $- 376811785879552m_8 + 2y$} \\ \hline
        $18$ & \thead{$P_{16}-(1672327810756374525600m_1$ \\ $+1638168449212286640000m_2$ \\ $+2459231585909344224000m_3$ \\ $+3027100644801192960000m_4$ \\ $+3521903018928053760000m_5$ \\ $+2110315614326599680000m_6$ \\ $+921056151928872960000m_7$ \\ $+2676192208994304000000m_8$ \\ $1605715325396582400000m_9)x^{20}$} & \thead{$23883876282090684m_1$ \\
        $+ 22976196043040576m_2$ \\
        $+ 33941981926660352m_3$ \\
        $+42536767410890240m_4$ \\ 
        $+ 50188301664458752m_5$ \\
        $+ 29300667700887552m_6$ \\ 
        $+ 12409659905671168m_7$ \\
        $+ 37610704241885184m_8 $ \\ 
        $ + 23998307331473408m_9 + 2y$} \\ \hline
    \end{tabular}
    \caption{\centering Total Pontryagin Classes and Surgery Obstructions of \\ $\xi=m_1\xi_1+\dots+m_{t_{n,4}^{'}}\xi_{t_{n,4}^{'}}+y\xi_{S^4}\in[Th(\cp^n\times D^{4}),G/O]_{free}$}
    \label{signobstrtable2}
\end{table}

\newpage

\begin{table}[ht!]
    \centering
    \begin{tabular}{|c|c|c|}
        \hline 
        $n$: & Total Pontryagin class & Surgery obstruction \\ \hline
        $1$ & \thead{$1-1440m_1x^4$} & \thead{$28m_1$} \\ \hline
        $3$ & \thead{$P_{1}-1440m_1x^4+(100800m_1+120960m_2)x^6$} & \thead{$-752m_1 - 992m_2$} \\ \hline
        $5$ & \thead{$P_{3}-(3669120m_1+$ \\ $+3628800m_2+2419200m_3)x^8$} & \thead{$10760m_1 + 10208m_2 + 8128m_3$} \\ \hline
        $7$ & \thead{$P_{5}-2419200m_3x^8+(132485760m_1$ \\ $+132523776m_2+127733760m_3+191600640m_4)x^{10}$} & \thead{$-150368m_1 - 151360m_2 - 131072m_3 - 261632m_4$} \\ \hline
        $9$ & \thead{$P_{7}-(2987254670400m_1+1443025584000m_2$ \\ $+4395238848000m_3+435891456000m_4$ \\ $+2615348736000m_5)x^{12}$} & \thead{$1653845080m_1 + 798623296m_2 + 2433654016m_3$ \\ $ + 240531968m_4 + 1448424448m_5$} \\ \hline
        $11$ & \thead{$P_{9}-2615348736000m_5x^{12}$ \\ $+(233012623072320m_1+112557132724608m_2$ \\ $+342707806172160m_3+33799023498240m_4$ \\ $+408143851315200m_5+298896998400m_6)x^{14}$} & \thead{$-45693629776m_1 - 22072942112m_2$ \\
        $- 67200433152m_3 - 6623920640m_4$ \\ $- 80038572032m_5 - 67100672m_6$} \\ \hline
        $13$ & \thead{$P_{11}-(33739555437907200m_1$ \\ $+36805134614976000m_2+44800972280064000m_3$ \\ $+41654949912576000m_4+54064489070592000m_5$ \\ $+3556874280960000m_6+21341245685760000m_7)x^{16}$} & \thead{$2842001070888m_1 + 3238745624352m_2 $ \\ $ + 3741180906624m_3 + 3757064812032m_4$ \\
        $+4520035100672m_5 + 323320668160m_6$ \\ $+ 1941802827776m_7$} \\ \hline
        $15$ & \thead{$P_{13}-21341245685760000m_7x^{16}$ \\ $+(8402144205069861120m_1$ \\
        $+22393402585339516416m_2$ \\ $+23812982725176852480m_3$ \\
        $+26773106681119518720m_4$ \\ $+26105038547735347200m_5$ \\
        $+18960416317056614400m_6$ \\ $+13624251245789184000m_7$ \\
        $+20436376868683776000m_8)x^{18}$} & \thead{$-294116919901376m_1 - 808244271451264m_2$ \\ $ - 857352847937536m_3 - 967180422607872m_4 $ \\
        $- 937562552582144m_5 - 697470026547200m_6$ \\ $ - 481703151009792m_7 - 753623571759104m_8$} \\ \hline
        $17$ & \thead{$P_{15}-(2884008786518448388800m_1$ \\ $+6395040144075361680000m_2$ \\ $+5271131206641047616000m_3$ \\ $+6181543187483279616000m_4$ \\ $+5708357806073674752000m_5$ \\ $+5146731790499819520000m_6$ \\ $+2368430104959959040000m_7$ \\ $+5620003638888038400000m_8$ \\ $+1605715325396582400000m_9)x^{20}$} & \thead{$41293608301111928m_1$ \\
        $+ 90677583712409216m_2$ \\
        $+ 73576727868219904m_3$ \\
        $+ 86525504491461632m_4$ \\ 
        $ + 79617778480623616m_5$ \\
        $+ 72730964263600128m_6$ \\
        $+ 32447476880769024m_7$ \\
        $+ 79472334229602304m_8$ \\ 
        $ + 23998307331473408m_9$} \\ \hline
    \end{tabular}
    \caption{\centering Total Pontryagin Classes and Surgery Obstructions of \\ $\xi=m_1\xi_1+\dots+m_{t_{n,6}^{'}}\xi_{t_{n,6}^{'}}\in[Th(\cp^n\times D^{6}),G/O]_{free}$}
    \label{signobstrtable3}
\end{table}